\pdfoutput=1
\documentclass[11pt]{article}

\usepackage[a4paper,margin=28mm]{geometry}

\usepackage[T1]{fontenc}
\usepackage[utf8]{inputenc}
\usepackage{lmodern}
\usepackage{amsmath,amssymb,amsthm,mathtools}
\numberwithin{equation}{section}
\usepackage{bm}
\usepackage{mathrsfs}
\usepackage{comment}

\usepackage{graphicx}
\usepackage{booktabs}
\usepackage{array}
\usepackage{enumitem}
\usepackage{tikz}
\usetikzlibrary{arrows.meta,calc,positioning}

\usepackage[colorlinks=true,linkcolor=blue,citecolor=blue,urlcolor=blue]{hyperref}
\usepackage[nameinlink,capitalize]{cleveref}

\theoremstyle{plain}
\newtheorem{theorem}{Theorem}[section]
\newtheorem{lemma}[theorem]{Lemma}
\newtheorem{proposition}[theorem]{Proposition}
\newtheorem{corollary}[theorem]{Corollary}

\theoremstyle{definition}
\newtheorem{definition}[theorem]{Definition}

\newtheorem{example}[theorem]{Example}

\theoremstyle{remark}
\newtheorem{remark}[theorem]{Remark}

\newcommand{\R}{\mathbb{R}}

\newcommand{\diff}{\mathop{}\!\mathrm{d}}

\newcommand{\Span}{\operatorname{span}}

\newcommand{\Sph}{\mathbb{S}}

\newcommand{\aff}{\operatorname{aff}}
\newcommand{\dist}{\operatorname{dist}}
\newcommand{\conv}{\operatorname{conv}}

\theoremstyle{definition}
\newtheorem{convention}[theorem]{Convention}

\title{Exact characterisation of maximum-angle conditions for spherical finite element meshes}

\author{
  Hiroki Ishizaka\\
  Team FEM, Matsuyama, Japan\\
  E-mail: \texttt{h.ishizaka005@gmail.com}
}

\date{}

\begin{document}
\maketitle

\begin{abstract}
Maximum-angle conditions are standard finite-element mesh hypotheses that permit anisotropic triangles excluded by minimum-angle or shape-regularity assumptions. For exact spherical triangles, however, the angles of the chordal affine core and the intrinsic spherical angles need not coincide, while radial geometry introduces curvature-scale distortion. We give an algebraic characterisation of the intrinsic spherical maximum-angle condition through a dimensionless quantity computed from the three vertex vectors. A uniform positive lower bound on this quantity is equivalent to a uniform spherical maximum-angle bound and requires neither spherical-angle nor spherical-area evaluation. We identify the support-plane geometry linking the chordal circumradius to radial distortion and derive sharp comparisons between the intrinsic spherical and chordal semi-regularity parameters. In particular, a locality-independent comparison holds with sharp constant $2/\sqrt3$ and a characterised equality case. An area-based parameter is shown to be smaller than the spherical semi-regularity parameter, with sharp constant one in the local flat limit. For spherical finite-element meshes, the vertex criterion implies uniform chordal semi-regularity with an explicit constant, while relative refinement makes the radial-distortion factors converge uniformly to one. Thus, the intrinsic criterion, together with relative refinement, provides the geometric controls used in anisotropic finite-element analysis without imposing a minimum-angle or shape-regularity condition.
\end{abstract}
\medskip
\noindent\textbf{Keywords.}\\
Maximum-angle condition; semi-regularity condition; exact spherical triangle;\\
spherical finite-element mesh; radial projection.

\medskip
\noindent\textbf{2020 Mathematics Subject Classification.}
Primary 65N30; Secondary 51M09.

\section{Introduction}
\label{sec:introduction}
Mesh geometry enters finite-element analysis through interpolation, stability and approximation constants. In two dimensions, the classical minimum-angle condition is stronger than is needed for many purposes. Synge~\cite{Synge1957} identified the largest angle as the relevant quantity for linear interpolation, and Babu\v ska and Aziz~\cite{BabuskaAziz1976} established the maximum-angle condition in finite-element analysis. Related formulations were developed by Jamet~\cite{Jamet1976} and K\v r\'i\v zek~\cite{Krizek1991}, and the same principle appears in anisotropic, mixed and nonconforming approximation \cite{ApelDobrowolski1992,Apel1999,AcostaDuran1999}. Its main geometric advantage is that a minimum angle may tend to zero while the largest angle remains uniformly separated from $\pi$. 

The anisotropic interpolation framework developed in \cite{IshizakaKobayashiTsuchiya2021,IshizakaKobayashiTsuchiya2023} is based on the semi-regularity condition employed in those works, which is equivalent in two dimensions to the maximum-angle condition and allows strongly anisotropic elements. Under this condition, anisotropic interpolation error estimates have been established for Lagrange and Crouzeix--Raviart elements \cite{IshizakaKobayashiTsuchiya2021,IshizakaKobayashiTsuchiya2023} and for Raviart--Thomas interpolation~\cite{Ishizaka2022RT}. In three dimensions, the equivalence between the semi-regularity condition and the maximum-angle condition for tetrahedra was established in \cite{IshizakaKobayashiSuzukiTsuchiya2021}.

For curved surface elements, an additional geometric issue arises. The angles of the planar triangle spanned by the vertices need not coincide with the intrinsic surface angles, and the map from the planar core to the surface may introduce additional distortion. Surface and curved finite-element analyses therefore employ geometric assumptions adapted to the discrete surface or the element map \cite{CiarletRaviart1972,Ciarlet1978,Bernardi1989,Dziuk1988,Demlow2009,DziukElliott2013}. In the exact-curved setting, Ishizaka~\cite{Ishizaka2026ExactCurved} used a two-stage element map that separates the affine core from the curved correction, allowing affine scaling and curvature effects to be analysed separately. Maximum-angle properties have also been studied for piecewise-planar surface meshes~\cite{OlshanskiiReuskenXu2013}, while finite-element methods on spherical geodesic grids have a long history~\cite{Giraldo1997}. The setting considered here is more specific: each element is an exact geodesic spherical triangle.

Let $K\subset\Sph_R^2$ have vertices $p_1,p_2,p_3$, with their convex hull as its chordal affine core. Chordal and geodesic side lengths have the same ordering, but the corresponding chordal and spherical interior angles need not agree at finite curvature scale. Furthermore, the supporting plane of the chordal core may approach the sphere centre and produce strong radial distortion even when the planar triangle has good shape. The main question is therefore whether the intrinsic spherical maximum-angle condition can be characterised directly from the vertex coordinates without imposing a minimum-angle condition.

The following result answers this question. If $p_1$ is opposite a longest geodesic edge and $\alpha_K$ is the largest spherical interior angle, then
\begin{align}
  \frac{
    R\,|\det(p_1,p_2,p_3)|
  }{
    |p_1\times p_2|\,|p_1\times p_3|
  }
  =
  \sin\alpha_K.
  \label{eq:intro-determinant-criterion}
\end{align}
Thus, the intrinsic spherical maximum-angle condition is exactly equivalent to a uniform positive lower bound for a quantity computed directly from the vertex vectors. The left-hand side is invariant under a rescaling of the sphere radius, so $R$ may be removed by normalising the three vertices. No spherical angle, local chart or spherical-area computation is needed. We emphasise that we do not impose the minimum-angle or the shape-regularity condition on the mesh partition.

We also relate this intrinsic criterion to the chordal and radial geometry. The geometric length scale of the chordal core is exactly four times its Euclidean circumradius, and its circumradius, support-plane distance and radial-distortion factor satisfy exact identities. The intrinsic spherical semi-regularity parameter is then compared with the circumradius parameter, which coincides with the chordal semi-regularity parameter. In particular, the global comparison has the sharp constant $2/\sqrt3$ and requires no spherical-locality assumption. Therefore, spherical semi-regularity implies chordal semi-regularity globally. The reverse implication fails for coarse families, but holds under uniform spherical locality. A sharper locality-dependent comparison also identifies the optimal unit-prefactor power $1/2$ of the radial-distortion factor.

We also introduce an area-based intrinsic parameter. It is strictly smaller than the spherical semi-regularity parameter, with sharp constant one in the local flat limit, and admits two-sided bounds in terms of the circumradius parameter and powers of the radial-distortion factor. In these comparisons, the spherical semi-regularity parameter measures maximum-angle degeneration, whereas the radial-distortion factor measures the ratio of the sphere radius to the distance from the sphere centre to the support plane of the chordal core.

For spherical finite-element meshes, a uniform positive lower bound for the vertex factor gives uniform chordal semi-regularity with an explicit constant. Under relative refinement $q_h\to0$, the radial-locality parameters tend uniformly to zero and the radial-distortion factors converge uniformly to one. The resulting chordal semi-regularity and radial-distortion control are associated, respectively, with the affine and radial stages of the parametrisation from a reference triangle to an exact spherical element. The present paper is restricted to this geometric setting. Anisotropic Lagrange interpolation on exact spherical triangles is left to a future study.

The remainder of the paper is organised as follows. Section~\ref{sec:geometry} develops the geometry of exact spherical triangles. Section~\ref{sec:angles} establishes the vertex-based maximum-angle characterisation, and Section~\ref{sec:geometric-comparisons} develops the comparison theory. Section~\ref{sec:fem-interpretation} presents the finite-element mesh consequences, and Section~\ref{sec:conclusion} concludes the paper.

\section{Geometry of an exact spherical triangle}
\label{sec:geometry}
This section introduces an exact spherical triangle together with its chordal affine core.  We use the same factorised viewpoint as in the exact-curved finite-element framework of \cite{Ishizaka2026ExactCurved}. The affine core provides an auxiliary Euclidean representation, whereas the radial correction carries the spherical geometry.  The intrinsic spherical semi-regularity condition itself will be defined directly from the curved triangle $K$ in Section~\ref{sec:angles}. The chordal quantities introduced here serve as exact comparison and finite-element bridge parameters.

\subsection{The sphere and geodesic distance}
\label{subsec:spherical-calculus}
Let $R>0$ and define
\begin{align*}
  \Sph_R^2
  :=
  \{x\in\R^3:|x|=R\}.
\end{align*}
For $x\in\Sph_R^2$, we write
\begin{align*}
  n(x)
  :=
  \frac{x}{R},
  \quad
  T_x\Sph_R^2
  :=
  \{\xi\in\R^3:\xi\cdot n(x)=0\}.
\end{align*}
The geodesic distance on $\Sph_R^2$ is
\begin{align*}
  d_{\Sph}(x,y)
  :=
  R\arccos\!\left(\frac{x\cdot y}{R^2}\right),
  \quad
  x,y\in\Sph_R^2.
\end{align*}
If $x\neq-y$, the minimizing geodesic joining $x$ and $y$ is the unique shorter great-circle arc. A subset $A\subset\Sph_R^2$ is called \emph{geodesically convex} if it contains no pair of antipodal points and contains the shorter great-circle arc joining every pair of its points.

\subsection{Construction of an exact spherical triangle}
Let $p_1,p_2,p_3\in\Sph_R^2$ be non-collinear and set
\begin{align*}
  T:=\conv\{p_1,p_2,p_3\}\subset\R^3.
\end{align*}
For a chordal triangle $T$, we denote its Euclidean area by $|T|_2$. We assume
\begin{align}
  0\notin\aff(T),
  \label{eq:non-great-circle-assumption}
\end{align}
equivalently $\det(p_1,p_2,p_3)\neq0$. Thus, the vertices do not lie on a common great circle. Let $\nu_T$ be the unit normal oriented so that
\begin{align}
  \aff(T)=\{y\in\R^3:\nu_T\cdot y=d_T\},
  \quad
  d_T:=\dist(0,\aff(T))>0.
  \label{eq:support-plane-normalisation}
\end{align}
We set $c_T:=d_T\nu_T$ and $V_T:=\nu_T^\perp$. Then, $\aff(T)=c_T+V_T$, $\dim V_T=2$, and $\nu_T\cdot y=d_T>0$ for any $y\in T$. Furthermore, we define the radial map and the associated exact spherical triangle as
\begin{align}
  \Psi_T(y)&:=R\frac{y}{|y|},
  \label{eq:radial-projection}\\
  K&:=\Psi_T(T).
  \label{eq:exact-spherical-element}
\end{align}
We denote the spherical surface area of $K$ by $|K|_2:=\mathcal H^2(K)$, where $\mathcal H^2$ denotes the two-dimensional Hausdorff measure in $\mathbb R^3$.

\begin{lemma}[Basic properties of the radial correction]
  \label{lem:radial-basic}
  The map $\Psi_T:T\to K$ is a bijection with inverse
  \begin{align}
    G_T(x)=\frac{d_T}{\nu_T\cdot x}x.
    \label{eq:radial-inverse-map}
  \end{align}
  Each chordal edge is mapped bijectively onto the corresponding shorter great-circle arc. Furthermore,
  \begin{align*}
    K\subset\mathcal H_T^+:=\{x\in\Sph_R^2:\nu_T\cdot x>0\},
  \end{align*}
  and $K$ is geodesically convex.
\end{lemma}

\begin{proof}
  If $\Psi_T(y_1)=\Psi_T(y_2)$, then $y_1=cy_2$ for some $c>0$. Taking the scalar product with $\nu_T$ gives $d_T=cd_T$, hence $c=1$. Thus, $\Psi_T$ is injective, and surjectivity follows from the definition of $K$. If $x=\Psi_T(y)$, then $\nu_T\cdot x=Rd_T/|y|$, which gives $G_T(x)=y$ and proves \eqref{eq:radial-inverse-map}.

  For an edge $e_{ij}=\conv\{p_i,p_j\}$, the points $p_i,p_j$ are not antipodal because $d_T>0$. Because $e_{ij}\subset\Span\{p_i,p_j\}$, radial projection maps it into the great circle through its endpoints. The image lies in $\mathcal H_T^+$, and continuity, injectivity and endpoint preservation show that it is precisely the shorter great-circle arc.

 Finally, extend $G_T$ to $\mathcal H_T^+$ by
\begin{align*}
  \widetilde G_T(x)
  :=
  \frac{d_T}{\nu_T\cdot x}x.
\end{align*}
Let $x_0,x_1\in K$. If $x_0=x_1$, there is nothing to prove. Otherwise, because $K\subset\mathcal H_T^+$, the points are not antipodal, and their central angle satisfies $0<\theta<\pi$. The shorter great-circle arc is
\begin{align*}
  x(s)=a_sx_0+b_sx_1,
  \quad
  a_s=\frac{\sin((1-s)\theta)}{\sin\theta},
  \quad
  b_s=\frac{\sin(s\theta)}{\sin\theta},
  \quad
  0\leq s\leq1.
\end{align*}
Because $a_s,b_s\geq0$, one has $\nu_T\cdot x(s)>0$. Writing $y_i=G_T(x_i)$ gives
\begin{align*}
  \widetilde G_T(x(s))
  =
  \frac{
    a_s(\nu_T\cdot x_0)y_0
    +
    b_s(\nu_T\cdot x_1)y_1
  }{
    a_s(\nu_T\cdot x_0)
    +
    b_s(\nu_T\cdot x_1)
  }
  \in T,
\end{align*}
because the right-hand side is a convex combination of $y_0$ and $y_1$. Since
\begin{align*}
  \Psi_T\bigl(\widetilde G_T(x(s))\bigr)=x(s),
\end{align*}
we obtain $x(s)\in K$. Therefore, $K$ is geodesically convex.
\end{proof}

The support-plane distance is also vertex-computable:
\begin{align}
  d_T
  =\frac{|\det(p_1,p_2,p_3)|}
         {|(p_2-p_1)\times(p_3-p_1)|}
  =\frac{|\det(p_1,p_2,p_3)|}{2|T|_2}.
  \label{eq:support-plane-distance-formula}
\end{align}

\subsection{Canonical labelling and planar angular parameters}
\label{subsec:planar-factorisation}

\begin{convention}[Canonical labelling of the chordal affine core]
  \label{conv:chordal-labelling}
  Choose a longest edge of $T$, denote its endpoints by $p_2$ and $p_3$, and let $p_1$ be the remaining vertex. For $i=1,2,3$, let $\beta_i$ denote the Euclidean interior angle of $T$ at $p_i$. Since $p_2p_3$ is a longest edge, its opposite angle is maximal, and we set
  \begin{align*}
    \beta_T
    :=
    \beta_1
    =
    \max_{1\leq i\leq3}\beta_i.
  \end{align*}
  Interchanging $p_2$ and $p_3$ if necessary, assume
  \begin{align*}
    |p_1-p_3|
    \leq
    |p_1-p_2|.
  \end{align*}
  We then set
  \begin{align*}
    h_T
    &:=
    |p_2-p_3|,
    \quad
    h_{T,1}
    :=
    |p_1-p_2|,
    \quad
    h_{T,2}
    :=
    |p_1-p_3|.
  \end{align*}
\end{convention}

The convention gives
\begin{align*}
  0<h_{T,2}\leq h_{T,1}\leq h_T<2h_{T,1},
\end{align*}
where the last inequality follows from the strict triangle inequality. If a longest edge or the ordering of the remaining two edges is not unique, any admissible labelling may be chosen; the geometric quantities and estimates below are unaffected by such a choice. This convention imposes no additional geometric restriction on $T$ and is consistent with the labelling used in the two-dimensional anisotropic interpolation theory of \cite{IshizakaKobayashiTsuchiya2021,IshizakaKobayashiTsuchiya2023}.

\subsection{The support plane and the radial distortion factor}
\label{subsec:radial-distortion}
We introduce only the first-order radial geometry needed later for the area comparison. We consider the ambient radial map
\begin{align*}
  \Psi:\R^3\setminus\{0\}\longrightarrow\Sph_R^2,
  \quad
  \Psi(y):=R\frac{y}{|y|},
\end{align*}
so that $\Psi_T=\Psi|_T$. For $y\in T$, we set
\begin{align*}
  x:=\Psi_T(y),
  \quad
  \hat y:=\frac{y}{|y|}=\frac{x}{R},
  \quad
  P_y:=I-\hat y\hat y^{\top}.
\end{align*}
Let $D\Psi(y):\R^3\to T_x\Sph_R^2$ be the Fr\'echet derivative of the ambient radial map $\Psi$ at $y$. Then, $P_y$ is the orthogonal projection onto $T_x\Sph_R^2$, and
\begin{align}
  D\Psi(y)\xi
  =
  \frac{R}{|y|}P_y\xi,
  \quad
  \xi\in\R^3.
  \label{eq:radial-first-derivative}
\end{align}
Furthermore,
\begin{align}
  d_T\leq |y|\leq R
  \quad
  \forall y\in T,
  \label{eq:radial-radius-range}
\end{align}
and we define the dimensionless radial distortion factor
\begin{align}
  \chi_T:=\frac{R}{d_T}>1.
  \label{eq:radial-distortion-factor}
\end{align}
The lower bound in \eqref{eq:radial-radius-range} follows from $\nu_T\cdot y=d_T$, whereas the upper bound follows from the convexity of the Euclidean norm. Because the three vertices are non-collinear, $\aff(T)$ cannot be tangent to $\Sph_R^2$; hence $d_T<R$ and $\chi_T>1$.

\begin{lemma}[Metric and area distortion of the radial correction]
  \label{lem:metric-distortion}
  Let $y\in T$, set $r:=|y|$, and restrict $D\Psi(y)$ to the chordal plane direction space $V_T$. Its two singular values are
  \begin{align}
    \sigma_{\min}(y)
    =\frac{Rd_T}{r^2},
    \quad
    \sigma_{\max}(y)
    =\frac{R}{r}.
    \label{eq:radial-singular-values}
  \end{align}
  Consequently,
  \begin{align}
    \chi_T^{-1}|\xi|
    \leq
    |D\Psi(y)\xi|
    \leq
    \chi_T|\xi|
    \quad
    \forall\xi\in V_T,
    \label{eq:radial-metric-bounds}
  \end{align}
  and the surface Jacobian of $\Psi_T$ with respect to the Euclidean area on $T$ and the spherical surface area on $K$ is
  \begin{align}
    J_T(y)
    =\frac{R^2d_T}{|y|^3},
    \quad
    \chi_T^{-1}\leq J_T(y)\leq\chi_T^2.
    \label{eq:radial-jacobian-bounds}
  \end{align}
\end{lemma}

\begin{proof}
  We write
  \begin{align*}
    y=d_T\nu_T+z,
    \quad z\in V_T.
  \end{align*}
  If $z\neq0$, we choose an orthonormal basis of $V_T$ given as
  \begin{align*}
    e_{\mathrm{rad}}:=\frac{z}{|z|},
    \quad
    e_{\mathrm{tan}}\perp e_{\mathrm{rad}}.
  \end{align*}
  Because $y\cdot e_{\mathrm{tan}}=0$,
  \begin{align*}
    P_y e_{\mathrm{tan}}=e_{\mathrm{tan}}.
  \end{align*}
  Furthermore,
  \begin{align*}
    |P_y e_{\mathrm{rad}}|^2
    =
    1-\frac{|z|^2}{r^2}
    =
    \frac{d_T^2}{r^2},
  \end{align*}
  and
  \begin{align*}
    (P_y e_{\mathrm{rad}})
    \cdot
    (P_y e_{\mathrm{tan}})
    =0.
  \end{align*}
  Therefore, the images of this orthonormal basis are orthogonal and have lengths $d_T/r$ and $1$. Thus, the singular values of $P_y|_{V_T}$ are $d_T/r$ and $1$. If $z=0$, then $r=d_T$ and $P_y|_{V_T}=I_{V_T}$, so the same conclusion holds. Multiplication by the factor $R/r$ in \eqref{eq:radial-first-derivative} shows that the singular values of $D\Psi(y)|_{V_T}$ are
  \begin{align*}
    \frac{Rd_T}{r^2},
    \quad
    \frac{R}{r},
  \end{align*}
  which proves \eqref{eq:radial-singular-values}.

  Because $d_T\leq r\leq R$,
  \begin{align*}
    \frac{Rd_T}{r^2}
    \geq
    \frac{d_T}{R}
    =
    \chi_T^{-1},
    \quad
    \frac{R}{r}
    \leq
    \frac{R}{d_T}
    =
    \chi_T.
  \end{align*}
  Therefore, \eqref{eq:radial-metric-bounds} follows.

  The surface Jacobian is the product of these two singular values; therefore,
  \begin{align*}
    J_T(y)
    =
    \frac{R^2d_T}{r^3}.
  \end{align*}
  Finally, using $d_T\leq r\leq R$ gives
  \begin{align*}
    \frac{d_T}{R}
    \leq
    J_T(y)
    \leq
    \frac{R^2}{d_T^2},
  \end{align*}
  that is,
  \begin{align*}
    \chi_T^{-1}
    \leq
    J_T(y)
    \leq
    \chi_T^2.
  \end{align*}
\end{proof}

\subsection{The chordal maximum-angle parameter}
Under Convention~\ref{conv:chordal-labelling}, define
\begin{align}
  H_T
  &:=
  \frac{h_{T,1}h_{T,2}}{|T|_2}h_T,
  \quad
  \mu_T
  :=
  \frac{H_T}{h_T}
  =
  \frac{h_{T,1}h_{T,2}}{|T|_2}.
  \label{eq:chordal-geometric-parameters}
\end{align}
The parameter $H_T$ was introduced to obtain anisotropic interpolation estimates with constants that depend explicitly on simplex geometry; see \cite{IshizakaKobayashiTsuchiya2021,IshizakaKobayashiTsuchiya2023}.

Because
\begin{align*}
  |T|_2
  =
  \frac12h_{T,1}h_{T,2}\sin\beta_T,
\end{align*}
we have
\begin{align}
  \mu_T
  =
  \frac{2}{\sin\beta_T},
  \quad
  H_T
  =
  4\mathcal R_T,
  \label{eq:mu-circumradius-identities}
\end{align}
where $\mathcal R_T$ is the circumradius of $T$; the second identity follows from the extended law of sines. Thus, $\mu_T\ge2$, and a uniform bound for $\mu_T$ is exactly the maximum-angle condition on the chordal affine cores.

\begin{definition}[Chordally semi-regular family]
  \label{def:chordal-semiregularity}
  Let $\{\mathscr K_h\}_{h>0}$ be a family of exact spherical triangles and let $T_K$ be the chordal affine core associated with $K\in\mathscr K_h$. The family is called \emph{chordally semi-regular} if there exists $\gamma_0<\infty$, independent of $h$ and $K$, such that
  \begin{align}
    \mu_{T_K}\leq\gamma_0
    \quad
    \forall K\in\mathscr K_h.
    \label{eq:chordal-semi-regular-condition}
  \end{align}
\end{definition}

\begin{remark}[Maximum-angle versus minimum-angle conditions]
  \label{rem:maximum-vs-minimum-angle}
  Because $\mu_T=4\mathcal R_T/h_T$, chordal semi-regularity is a circumradius-to-diameter, equivalently maximum-angle, condition. It is strictly weaker than usual shape regularity. For example, triangles with angles $\varepsilon$, $\pi/2-\varepsilon$ and $\pi/2$ remain chordally semi-regular while their minimum angle tends to zero. After a suitable similarity scaling, such triangles can also be realised as chordal cores of exact spherical triangles.
\end{remark}

\subsection{Spherical locality associated with the chordal bridge parameter}
The chordal parameter $\mu_T=H_T/h_T$ controls the maximum-angle shape of the affine core, whereas
\begin{align*}
  \chi_T:=\frac{R}{d_T}
\end{align*}
measures the radial distortion associated with its supporting plane. Recall that $c_T=d_T\nu_T$ is the orthogonal projection of the origin onto $\aff(T)$ and that $H_T/4=\mathcal R_T$ is the circumradius of $T$. The following identity relates these two geometric mechanisms exactly.

\begin{theorem}[Exact support-plane identity]
  \label{thm:support-plane-localisation}
  For $i=1,2,3$,
  \begin{align}
    |p_i-c_T|
    =
    \mathcal R_T
    =
    \frac{H_T}{4}.
    \label{eq:support-plane-vertex-distance}
  \end{align}
  Consequently,
  \begin{align}
    d_T^2+\mathcal R_T^2
    &=
    R^2,
    \quad
    d_T^2
    =
    R^2-\frac{H_T^2}{16},
    \quad
    0<H_T<4R,
    \quad
    \chi_T
    =
    \left(
      1-\frac{H_T^2}{16R^2}
    \right)^{-1/2}.
    \label{eq:support-plane-localisation}
  \end{align}
\end{theorem}

\begin{proof}
  Because $p_i\in\aff(T)$ and $c_T=d_T\nu_T$,
  \begin{align*}
    \nu_T\cdot(p_i-c_T)=0.
  \end{align*}
  Therefore, $p_i-c_T\in V_T$, and
  \begin{align*}
    |p_i-c_T|^2
    =
    |p_i-d_T\nu_T|^2
    =
    R^2-d_T^2.
  \end{align*}
  Thus, the three non-collinear vertices lie on the circle $\aff(T)\cap\Sph_R^2$ centred at $c_T$. This is the circumcircle of $T$, so
  \begin{align*}
    |p_i-c_T|
    =
    \mathcal R_T
    =
    \frac{H_T}{4}.
  \end{align*}
  The orthogonal decomposition $p_i=c_T+(p_i-c_T)$ therefore gives
  \begin{align*}
    d_T^2+\mathcal R_T^2=R^2.
  \end{align*}
  Using $H_T=4\mathcal R_T$ gives the remaining identities. Because $\mathcal R_T>0$ and $d_T>0$, one has $0<H_T<4R$.
\end{proof}

\subsection{Intrinsic circumradius geometry and spherical locality}
\label{subsec:intrinsic-circumradius}
For $i\neq j$, we define the central angle associated with the geodesic edge joining $p_i$ and $p_j$ as
\begin{align}
  \theta_{ij}
  &:=
  \frac{d_{\Sph}(p_i,p_j)}{R}
  =
  \arccos\!\left(\frac{p_i\cdot p_j}{R^2}\right)
  \in(0,\pi),
  \label{eq:spherical-edge-angle}
\end{align}
and we set
\begin{align}
  \theta_K
  :=
  \max_{1\leq i<j\leq3}\theta_{ij}.
  \label{eq:largest-spherical-side-angle}
\end{align}
We also introduce the dimensionless chordal locality parameter
\begin{align}
  \vartheta_T
  :=
  \frac{\mathcal R_T}{R}
  =
  \frac{H_T}{4R}.
  \label{eq:chordal-locality-parameter}
\end{align}
We set
\begin{align}
  q_K:=R\nu_T\in\Sph_R^2.
  \label{eq:spherical-circumcentre}
\end{align}
Because $\nu_T\cdot p_i=d_T$ for $i=1,2,3$,
\begin{align*}
  \frac{q_K\cdot p_i}{R^2}
  =
  \frac{d_T}{R},
\end{align*}
so the three vertices are equidistant from $q_K$. We therefore define their common angular distance from $q_K$ as
\begin{align}
  \omega_K
  :=
  \frac{d_{\Sph}(q_K,p_i)}{R}
  =
  \arccos\!\left(\frac{d_T}{R}\right)
  \in\left(0,\frac{\pi}{2}\right).
  \label{eq:spherical-angular-circumradius}
\end{align}
The point $q_K$ is the spherical circumcentre with the smaller angular circumradius. Indeed, if $q\in\Sph_R^2$ is equidistant from the three vertices, then
\begin{align*}
  q\cdot(p_2-p_1)
  =
  q\cdot(p_3-p_1)
  =
  0.
\end{align*}
Because $p_2-p_1$ and $p_3-p_1$ span $V_T$, one has $q\in V_T^\perp=\Span\{\nu_T\}$ and hence $q=\pm R\nu_T$. The second circumcentre $-q_K$ has angular radius $\pi-\omega_K>\pi/2$.

\begin{theorem}[Intrinsic circumradius representation]
  \label{thm:intrinsic-circumradius-representation}
  Let $K$ be a non-degenerate exact spherical triangle and let $T$ be its chordal affine core. Then,
  \begin{align}
    \cos\omega_K
    &=
    \frac{d_T}{R},
    \quad
    \sin\omega_K
    =
    \frac{\mathcal R_T}{R}
    =
    \frac{H_T}{4R},
    \label{eq:intrinsic-circumradius-support}
  \end{align}
  and
  \begin{align}
    h_T
    =
    2R\sin\frac{\theta_K}{2}.
    \label{eq:intrinsic-largest-side-chord}
  \end{align}
  We define
  \begin{align}
    \vartheta_K
    &:=
    \sin\omega_K,
    \quad
    \chi_K
    :=
    \sec\omega_K,
    \quad
    \mu_K^{\mathrm{circ}}
    :=
    2\,
    \frac{\sin\omega_K}{\sin(\theta_K/2)}.
    \label{eq:intrinsic-spherical-parameters}
  \end{align}
  Then, the exact bridge identities
  \begin{align}
    \vartheta_K
    =
    \vartheta_T,
    \quad
    \chi_K
    =
    \chi_T,
    \quad
    \mu_K^{\mathrm{circ}}
    =
    \mu_T
    \label{eq:intrinsic-chordal-bridge}
  \end{align}
  hold without any locality assumption.
\end{theorem}

\begin{proof}
  The identity $\cos\omega_K=d_T/R$ follows from the definition of $\omega_K$. From Theorem~\ref{thm:support-plane-localisation},
  \begin{align*}
    d_T^2+\mathcal R_T^2=R^2,
    \quad
    \mathcal R_T=\frac{H_T}{4}.
  \end{align*}
  Because $\omega_K\in(0,\pi/2)$,
  \begin{align*}
    \sin\omega_K
    =
    \sqrt{1-\frac{d_T^2}{R^2}}
    =
    \frac{\mathcal R_T}{R}
    =
    \frac{H_T}{4R}.
  \end{align*}
  For $i\neq j$,
  \begin{align*}
    |p_i-p_j|
    =
    2R\sin\frac{\theta_{ij}}{2}.
  \end{align*}
  Because $\theta\mapsto2R\sin(\theta/2)$ is strictly increasing on $(0,\pi)$, the longest chordal and geodesic edges have the same endpoints. Therefore,
  \begin{align*}
    h_T
    =
    2R\sin\frac{\theta_K}{2}.
  \end{align*}
  It follows that
  \begin{align*}
    \mu_K^{\mathrm{circ}}
    =
    2\frac{H_T/(4R)}{h_T/(2R)}
    =
    \frac{H_T}{h_T}
    =
    \mu_T.
  \end{align*}
  The remaining bridge identities follow from $\vartheta_T=H_T/(4R)$ and $\chi_T=R/d_T$.
\end{proof}

\begin{definition}[Spherical locality]
  \label{def:spherical-locality}
  A family $\mathscr K$ of exact spherical triangles is called \emph{uniformly spherically local} if there exists $\eta_0\in(0,1)$ such that
  \begin{align}
    \vartheta_K
    =
    \sin\omega_K
    \leq
    \eta_0
    \quad
    \forall K\in\mathscr K.
    \label{eq:spherical-locality-condition}
  \end{align}
  Because
  \begin{align*}
    \chi_K
    =
    (1-\vartheta_K^2)^{-1/2},
  \end{align*}
  this is equivalent to the existence of $\chi_0>1$ such that $\chi_K\leq\chi_0$ uniformly. One may take
  \begin{align*}
    \chi_0=(1-\eta_0^2)^{-1/2},
    \quad
    \eta_0=\sqrt{1-\chi_0^{-2}}.
  \end{align*}
\end{definition}

\begin{corollary}[Automatic spherical locality under chordally semi-regular refinement]
  \label{cor:refining-locality}
  Let $\{\mathscr K_h\}_{h>0}$ be a family of exact spherical triangles whose chordal affine cores satisfy
  \begin{align*}
    \sup_h\max_{K\in\mathscr K_h}\mu_{T_K}
    \leq
    \gamma_0
    <
    \infty.
  \end{align*}
  We define the relative chordal diameter as
  \begin{align}
    q_h
    :=
    \max_{K\in\mathscr K_h}\frac{h_{T_K}}{R}.
    \label{eq:relative-chordal-diameter}
  \end{align}
  If $q_h\to0$, then
  \begin{align}
    \max_{K\in\mathscr K_h}\vartheta_K
    &\leq
    \frac{\gamma_0}{4}q_h
    \longrightarrow0, \quad
    \max_{K\in\mathscr K_h}\chi_K
    \longrightarrow1.
    \label{eq:uniform-locality-convergence}
  \end{align}
\end{corollary}

\begin{proof}
  From the bridge identities,
  \begin{align*}
    \vartheta_K
    =
    \frac{H_{T_K}}{4R}
    =
    \frac{\mu_{T_K}}{4}\frac{h_{T_K}}{R}
    \leq
    \frac{\gamma_0}{4}q_h.
  \end{align*}
  Therefore, $\max_K\vartheta_K\to0$. Because $s\mapsto(1-s^2)^{-1/2}$ is increasing on $[0,1)$,
  \begin{align*}
    \max_{K\in\mathscr K_h}\chi_K
    =
    \left(
      1-
      \bigl(\max_{K\in\mathscr K_h}\vartheta_K\bigr)^2
    \right)^{-1/2}
    \longrightarrow1.
  \end{align*}
\end{proof}

\begin{example}[Good chordal shape without spherical locality]
  \label{ex:coarse-equilateral}
  Let $\delta\in(0,R)$ and set
  \begin{align*}
    r_\delta:=\sqrt{R^2-\delta^2}.
  \end{align*}
  We choose an equilateral triangle on the circle $\Sph_R^2\cap\{x_3=\delta\}$. Then,
  \begin{align*}
    \mu_T=\frac{4}{\sqrt3},
    \quad
    h_T=\sqrt3\,r_\delta,
    \quad
    H_T=4r_\delta,
  \end{align*}
  whereas $d_T=\delta$. Consequently,
  \begin{align*}
    \vartheta_T
    =
    \frac{r_\delta}{R}
    \uparrow1,
    \quad
    \chi_T
    =
    \frac{R}{\delta}
    \longrightarrow\infty
    \quad
    (\delta\downarrow0).
  \end{align*}
  Thus, uniform chordal semi-regularity alone does not imply spherical locality for coarse elements. Indeed, $h_T/R\to\sqrt3$, so this family is not refining.
\end{example}

\begin{example}[Relative refinement and spherical locality do not imply chordal semi-regularity]
  \label{ex:refining-nonsemiregular}
  Let $\varepsilon\in(0,1)$ and set
  \begin{align*}
    r_\varepsilon:=R\varepsilon,
    \quad
    d_\varepsilon:=R\sqrt{1-\varepsilon^2}.
  \end{align*}
  We choose three points on the support circle in the plane $x_3=d_\varepsilon$ at polar angles $0,-\varepsilon,\varepsilon$, with $p_1$ corresponding to the angle $0$. Their chord lengths are
  \begin{align*}
    |p_2-p_3|
    =
    2R\varepsilon\sin\varepsilon,
    \quad
    |p_1-p_2|
    =
    |p_1-p_3|
    =
    2R\varepsilon\sin(\varepsilon/2).
  \end{align*}
  Thus, $p_2p_3$ is the longest edge, and the Euclidean angle opposite it is
  \begin{align*}
    \beta_\varepsilon
    =
    \pi-\varepsilon.
  \end{align*}
  Consequently,
  \begin{align*}
    h_{T_\varepsilon}
    =
    2R\varepsilon\sin\varepsilon
    \sim
    2R\varepsilon^2,
    \quad
    \mu_{T_\varepsilon}
    =
    \frac{2}{\sin\varepsilon}
    \sim
    \frac{2}{\varepsilon},
  \end{align*}
  while
  \begin{align*}
    H_{T_\varepsilon}
    =
    4R\varepsilon,
    \quad
    \vartheta_{T_\varepsilon}
    =
    \varepsilon,
    \quad
    \chi_{T_\varepsilon}
    =
    (1-\varepsilon^2)^{-1/2}.
  \end{align*}
  Therefore,
  \begin{align*}
    \frac{h_{T_\varepsilon}}{R}\to0,
    \quad
    \vartheta_{T_\varepsilon}\to0,
    \quad
    \chi_{T_\varepsilon}\to1,
  \end{align*}
  whereas
  \begin{align*}
    \mu_{T_\varepsilon}\to\infty,
    \quad
    \beta_\varepsilon\to\pi.
  \end{align*}
  Thus, relative refinement together with asymptotically vanishing radial distortion does not imply chordal semi-regularity; the degeneration is angular rather than radial.
\end{example}

\section{Vertex-based characterisation of spherical maximum angles}
\label{sec:angles}
The purpose of this section is to characterise the intrinsic spherical maximum-angle condition directly in terms of the three vertex vectors. The chordal affine core is used only to fix the canonical labelling. The resulting spherical semi-regularity condition is intrinsic to the curved triangle $K$.

\subsection{Spherical angles and side--angle ordering}
\label{subsec:spherical-angle-representation}
Let $i,j,k\in\{1,2,3\}$ be mutually distinct. Recall the central angles $\theta_{ij}$ from \eqref{eq:spherical-edge-angle}. Let $\alpha_i$ denote the spherical interior angle of $K$ at $p_i$, and set
\begin{align*}
  \alpha_K:=\max_{1\leq i\leq3}\alpha_i.
\end{align*}
Because $K$ is a non-degenerate geodesically convex spherical triangle,
\begin{align}
  0<\alpha_i<\pi,
  \quad
  i=1,2,3.
  \label{eq:spherical-angle-range}
\end{align}
The initial unit tangent at $p_i$ to the shorter great-circle arc directed towards $p_j$ is
\begin{align}
  t_{ij}
  :=
  \frac{
    p_j/R-\cos\theta_{ij}\,p_i/R
  }{
    \sin\theta_{ij}
  }.
  \label{eq:vertex-tangent-direction}
\end{align}
Because $t_{ij}\cdot p_i=0$ and $|t_{ij}|=1$, one has $t_{ij}\in T_{p_i}\Sph_R^2$. The spherical interior angle at $p_i$ is therefore characterised by
\begin{align}
  \cos\alpha_i
  =
  t_{ij}\cdot t_{ik}.
  \label{eq:spherical-angle-tangent}
\end{align}

We first introduce the spherical sine law in a form adapted to the vertex representation.

\begin{lemma}[Spherical sine law]
  \label{lem:spherical-sine-law}
  For mutually distinct $i,j,k\in\{1,2,3\}$,
  \begin{align}
    \frac{\sin\alpha_i}{\sin\theta_{jk}}
    =
    \frac{\sin\alpha_j}{\sin\theta_{ik}}
    =
    \frac{\sin\alpha_k}{\sin\theta_{ij}}.
    \label{eq:spherical-sine-law}
  \end{align}
\end{lemma}

\begin{proof}
  We set $n_i:=p_i/R$. From \eqref{eq:vertex-tangent-direction},
  \begin{align*}
    t_{ij}
    =
    \frac{
      n_j-(n_i\cdot n_j)n_i
    }{
      \sin\theta_{ij}
    }.
  \end{align*}
  Because $t_{ij},t_{ik}\in n_i^\perp$ are unit vectors forming the angle $\alpha_i$,
  \begin{align*}
    \sin\alpha_i
    =
    \left|
      n_i\cdot(t_{ij}\times t_{ik})
    \right|
    =
    \frac{
      |n_i\cdot(n_j\times n_k)|
    }{
      \sin\theta_{ij}\sin\theta_{ik}
    }.
  \end{align*}
  Therefore,
  \begin{align*}
    \frac{\sin\alpha_i}{\sin\theta_{jk}}
    =
    \frac{
      |\det(n_i,n_j,n_k)|
    }{
      \sin\theta_{ij}
      \sin\theta_{ik}
      \sin\theta_{jk}
    }.
  \end{align*}
  The right-hand side is invariant under permutations of $i,j,k$, which proves \eqref{eq:spherical-sine-law}.
\end{proof}

The chordal and geodesic side lengths have the same ordering. Indeed,
\begin{align}
  |p_i-p_j|
  =
  2R\sin\frac{\theta_{ij}}2,
  \label{eq:chord-geodesic-ordering}
\end{align}
and $\theta\mapsto2R\sin(\theta/2)$ is strictly increasing on $(0,\pi)$. It remains to relate the spherical side lengths to their opposite interior angles.

\begin{lemma}[Side--angle ordering]
  \label{lem:spherical-side-angle-ordering}
  For mutually distinct $i,j,k\in\{1,2,3\}$,
  \begin{align*}
    \theta_{jk}>\theta_{ik}
    \quad\Longleftrightarrow\quad
    \alpha_i>\alpha_j,
  \end{align*}
  and
  \begin{align*}
    \theta_{jk}=\theta_{ik}
    \quad\Longleftrightarrow\quad
    \alpha_i=\alpha_j.
  \end{align*}
  Consequently, a vertex opposite a longest spherical side realises a largest spherical interior angle.
\end{lemma}

\begin{proof}
  We put
  \begin{align*}
    a:=\theta_{jk},
    \quad
    b:=\theta_{ik},
    \quad
    c:=\theta_{ij},
    \quad
    A:=\alpha_i,
    \quad
    B:=\alpha_j.
  \end{align*}
  From Lemma~\ref{lem:spherical-sine-law}, there exists $\lambda>0$ such that
  \begin{align*}
    \sin a=\lambda\sin A,
    \quad
    \sin b=\lambda\sin B.
  \end{align*}
  The spherical cosine laws for the sides $a$ and $b$ give
  \begin{align*}
    \cos a
    =
    \cos b\cos c
    +
    \sin b\sin c\cos A,
  \end{align*}
  and
  \begin{align*}
    \cos b
    =
    \cos a\cos c
    +
    \sin a\sin c\cos B.
  \end{align*}
  Subtracting these identities gives
  \begin{align*}
    (1+\cos c)(\cos a-\cos b)
    &=
    \sin c
    \left(
      \sin b\cos A-\sin a\cos B
    \right)
    \\
    &=
    -\lambda\sin c\,\sin(A-B).
  \end{align*}
  Thus,
  \begin{align}
    (1+\cos c)(\cos a-\cos b)
    =
    -\lambda\sin c\,\sin(A-B).
    \label{eq:spherical-side-angle-order}
  \end{align}
  Because $c\in(0,\pi)$, both $1+\cos c$ and $\lambda\sin c$ are positive. Furthermore, cosine is strictly decreasing on $(0,\pi)$ and $A-B\in(-\pi,\pi)$. Therefore,
  \begin{align*}
    a>b
    \quad\Longleftrightarrow\quad
    A>B.
  \end{align*}
  If $a=b$, then \eqref{eq:spherical-side-angle-order} gives $\sin(A-B)=0$, hence $A=B$. Conversely, if $A=B$, the same identity gives $\cos a=\cos b$, and therefore $a=b$.
\end{proof}

By Convention~\ref{conv:chordal-labelling}, $p_1$ lies opposite a longest chordal edge. By \eqref{eq:chord-geodesic-ordering}, the same edge is a longest spherical side, and Lemma~\ref{lem:spherical-side-angle-ordering} therefore gives
\begin{align}
  \beta_1=\beta_T,
  \quad
  \alpha_1=\alpha_K.
  \label{eq:canonical-largest-angles}
\end{align}
This remains valid when a longest edge is not unique: the corresponding opposite angles are then equal and maximal.

Finally, the spherical excess formula gives
\begin{align*}
  \alpha_1+\alpha_2+\alpha_3
  =
  \pi+\frac{|K|_2}{R^2}
  >
  \pi.
\end{align*}
Together with \eqref{eq:spherical-angle-range}, this yields
\begin{align}
  \frac{\pi}{3}
  <
  \alpha_K
  <
  \pi.
  \label{eq:largest-spherical-angle-range}
\end{align}

\subsection{The vertex criterion and spherical semi-regularity}
\label{subsec:vertex-criterion}
The canonical vertex $p_1$ identified above allows the largest spherical angle to be characterised directly from the vertex vectors.

\begin{definition}[Spherical semi-regularity parameter]
  \label{def:spherical-semiregularity}
  We define the dimensionless algebraic factor
  \begin{align}
    t_K^{\mathrm{sr}}
    :=
    \frac{
      R\,|\det(p_1,p_2,p_3)|
    }{
      |p_1\times p_2|\,|p_1\times p_3|
    }.
    \label{eq:spherical-sr-factor}
  \end{align}
  Writing
  \begin{align*}
    n_i:=\frac{p_i}{R},
    \quad
    i=1,2,3,
  \end{align*}
  the same quantity has the radius-free representation
  \begin{align}
    t_K^{\mathrm{sr}}
    =
    \frac{
      |\det(n_1,n_2,n_3)|
    }{
      |n_1\times n_2|\,|n_1\times n_3|
    }.
    \label{eq:spherical-sr-factor-normalised}
  \end{align}
  We then define the \emph{spherical semi-regularity parameter} as
  \begin{align}
    \mu_K^{\mathrm{sr}}
    :=
    \frac{2}{t_K^{\mathrm{sr}}}.
    \label{eq:spherical-sr-parameter}
  \end{align}
\end{definition}

Since $K$ is contained in an open hemisphere, no two vertices are
antipodal. Hence the denominator in \eqref{eq:spherical-sr-factor} is
non-zero. Moreover, \eqref{eq:non-great-circle-assumption} gives
$\det(p_1,p_2,p_3)\neq0$, and therefore
\begin{align*}
  t_K^{\mathrm{sr}}>0.
\end{align*}
If more than one longest edge exists, any admissible choice of the opposite
vertex may be labelled $p_1$; Theorem~\ref{thm:spherical-sr-maximum-angle-identity}
below shows that the resulting value of $t_K^{\mathrm{sr}}$ is independent
of this choice.

\begin{theorem}[Exact algebraic characterisation of the spherical maximum angle]
  \label{thm:spherical-sr-maximum-angle-identity}
  \label{thm:spherical-maximum-angle-characterisation}
  For any non-degenerate exact spherical triangle,
  \begin{align}
    t_K^{\mathrm{sr}}
    =
    \sin\alpha_K,
    \quad
    \mu_K^{\mathrm{sr}}
    =
    \frac{2}{\sin\alpha_K}.
    \label{eq:spherical-sr-maximum-angle-identity}
  \end{align}
  This identity is exact and requires no spherical-locality assumption.
\end{theorem}

\begin{proof}
  We set
  \begin{align*}
    n_1:=\frac{p_1}{R},
    \quad
    P_1:=I-n_1n_1^{\top},
  \end{align*}
  so that $P_1$ is the orthogonal projection onto $T_{p_1}\Sph_R^2=n_1^\perp$. Let $t_{1m}$, $m=2,3$, denote the unit tangent directions from \eqref{eq:vertex-tangent-direction}. Because
  \begin{align*}
    p_m
    =
    R\left(
      \cos\theta_{1m}\,n_1
      +
      \sin\theta_{1m}\,t_{1m}
    \right),
  \end{align*}
  the tangent projection of $p_m$ at $p_1$ is
  \begin{align}
    \tau_{1m}
    :=
    P_1p_m
    =
    R\sin\theta_{1m}\,t_{1m},
    \quad
    m=2,3.
    \label{eq:tangent-projection-geodesic-direction}
  \end{align}
  Because $\theta_{1m}\in(0,\pi)$, the coefficients in \eqref{eq:tangent-projection-geodesic-direction} are positive. Therefore, by \eqref{eq:canonical-largest-angles}, the angle between $\tau_{12}$ and $\tau_{13}$ is precisely $\alpha_1=\alpha_K$. Furthermore,
  \begin{align*}
    |\tau_{1m}|^2
    &=
    R^2-\frac{(p_1\cdot p_m)^2}{R^2}
    =
    \frac{|p_1\times p_m|^2}{R^2},
  \end{align*}
  and hence
  \begin{align*}
    |\tau_{1m}|
    =
    \frac{|p_1\times p_m|}{R}.
  \end{align*}
  Because $\tau_{12},\tau_{13}\in n_1^\perp$, their cross product is parallel to $n_1$. Furthermore,
  \begin{align*}
    n_1\cdot(\tau_{12}\times\tau_{13})
    &=
    n_1\cdot(p_2\times p_3)
    =
    \frac{\det(p_1,p_2,p_3)}{R}.
  \end{align*}
  Therefore,
  \begin{align*}
    |\tau_{12}\times\tau_{13}|
    =
    \frac{|\det(p_1,p_2,p_3)|}{R}.
  \end{align*}
  Consequently,
  \begin{align*}
    \sin\alpha_K
    &=
    \frac{
      |\tau_{12}\times\tau_{13}|
    }{
      |\tau_{12}|\,|\tau_{13}|
    }
    =
    \frac{
      R\,|\det(p_1,p_2,p_3)|
    }{
      |p_1\times p_2|\,|p_1\times p_3|
    }
    =
    t_K^{\mathrm{sr}}.
  \end{align*}
  The second identity follows immediately from $\mu_K^{\mathrm{sr}}=2/t_K^{\mathrm{sr}}$.
\end{proof}

\begin{corollary}[Computable spherical semi-regularity and the maximum-angle condition]
  \label{cor:spherical-sr-maximum-angle-equivalence}
  Let $\mathscr K$ be a family of exact spherical triangles. Then, the following are equivalent:
  \begin{enumerate}[label=\textup{(\roman*)}]
    \item there exists $c_{\mathrm{sr}}>0$ such that
    \begin{align}
      t_K^{\mathrm{sr}}
      \geq
      c_{\mathrm{sr}}
      \quad
      \forall K\in\mathscr K;
      \label{eq:spherical-sr-lower-condition}
    \end{align}

    \item there exists $\gamma_{\mathrm{sr}}<\infty$ such that
    \begin{align}
      \mu_K^{\mathrm{sr}}
      \leq
      \gamma_{\mathrm{sr}}
      \quad
      \forall K\in\mathscr K;
      \label{eq:spherical-semiregularity-condition}
    \end{align}

    \item there exists $\delta>0$, independent of the element, such that
    \begin{align}
      \alpha_K
      \leq
      \pi-\delta
      \quad
      \forall K\in\mathscr K.
      \label{eq:spherical-maximum-angle-condition}
    \end{align}
  \end{enumerate}
  Thus, \eqref{eq:spherical-sr-lower-condition} is a vertex-only algebraic condition exactly equivalent to the intrinsic spherical maximum-angle condition.
\end{corollary}

\begin{proof}
  The equivalence of \textup{(i)} and \textup{(ii)} follows directly from
  \begin{align*}
    \mu_K^{\mathrm{sr}}
    =
    \frac{2}{t_K^{\mathrm{sr}}}.
  \end{align*}
  From Theorem~\ref{thm:spherical-sr-maximum-angle-identity},
  \begin{align*}
    t_K^{\mathrm{sr}}
    =
    \sin\alpha_K,
    \quad
    \mu_K^{\mathrm{sr}}
    =
    \frac{2}{\sin\alpha_K}.
  \end{align*}

  Suppose \textup{(ii)} holds. Because $\mu_K^{\mathrm{sr}}\geq2$, necessarily
  $\gamma_{\mathrm{sr}}\geq2$. Therefore,
  \begin{align*}
    \sin\alpha_K
    \geq
    \frac{2}{\gamma_{\mathrm{sr}}}.
  \end{align*}
  We set
  \begin{align*}
    \delta
    :=
    \arcsin\left(
      \frac{2}{\gamma_{\mathrm{sr}}}
    \right)
    \in
    \left(0,\frac{\pi}{2}\right].
  \end{align*}
  Because $0<\alpha_K<\pi$, the inequality $\sin\alpha_K\geq\sin\delta$ implies
  \begin{align*}
    \alpha_K
    \leq
    \pi-\delta.
  \end{align*}
  Thus \textup{(iii)} holds. Conversely, suppose \textup{(iii)} holds. Replacing $\delta$ by $\min\{\delta,\pi/2\}$ if necessary, we may assume $0<\delta\leq\pi/2$. From
  \begin{align*}
    \frac{\pi}{3}
    <
    \alpha_K
    \leq
    \pi-\delta,
  \end{align*}
  we obtain
  \begin{align*}
    \sin\alpha_K
    \geq
    \min\left\{
      \frac{\sqrt3}{2},
      \sin\delta
    \right\}.
  \end{align*}
  Therefore,
  \begin{align*}
    t_K^{\mathrm{sr}}
    =
    \sin\alpha_K
    \geq
    \min\left\{
      \frac{\sqrt3}{2},
      \sin\delta
    \right\}
    >0,
  \end{align*}
  proving \textup{(i)}.
\end{proof}

\begin{remark}[Direct mesh-quality test]
  \label{rem:direct-spherical-mesh-test}
  No spherical angle needs to be evaluated to test \eqref{eq:spherical-sr-lower-condition} or \eqref{eq:spherical-semiregularity-condition}. First identify a longest edge by comparing the three chord lengths $|p_i-p_j|$. Because chordal and geodesic side lengths have the same ordering, choose any vertex opposite a longest edge and label it $p_1$. Then, compute $|p_1\times p_2|$, $|p_1\times p_3|$, and $|\det(p_1,p_2,p_3)|$, and evaluate \eqref{eq:spherical-sr-factor}. Thus, the exact spherical maximum-angle condition can be checked directly from the vertex vectors, without evaluating inverse trigonometric functions or spherical areas.
\end{remark}

\section{Comparison of intrinsic and chordal geometric parameters}
\label{sec:geometric-comparisons}
We compare the exact spherical semi-regularity parameter $\mu_K^{\mathrm{sr}}$ with the intrinsic circumradius parameter $\mu_K^{\mathrm{circ}}=\mu_T$ and with an area-based intrinsic quantity. The comparison separates maximum-angle degeneration from radial locality and provides the geometric bridge to the chordal affine core.

\subsection{Exact spherical--chordal angle relations}
\label{subsec:exact-angle-relations}
We introduce the relations between the spherical interior angles and the Euclidean angles of the chordal affine core.

\begin{lemma}[Localisation of spherical side angles]
  \label{lem:side-localisation}
  For any pair $i\neq j$,
  \begin{align}
    \theta_{ij}
    \leq
    2\omega_K
    =
    2\arccos(\chi_K^{-1}),
    \label{eq:spherical-edge-localisation}
  \end{align}
  or equivalently,
  \begin{align}
    \cos\frac{\theta_{ij}}2
    \geq
    \cos\omega_K
    =
    \frac{d_T}{R}
    =
    \chi_K^{-1}.
    \label{eq:half-angle-locality-bound}
  \end{align}
\end{lemma}

\begin{proof}
  Because $q_K$ is the spherical circumcentre of $K$,
  \begin{align*}
    d_{\Sph}(q_K,p_i)
    =
    d_{\Sph}(q_K,p_j)
    =
    R\omega_K.
  \end{align*}
  The triangle inequality for the geodesic distance gives
  \begin{align*}
    R\theta_{ij}
    &=
    d_{\Sph}(p_i,p_j)
    \leq
    d_{\Sph}(p_i,q_K)
    +
    d_{\Sph}(q_K,p_j)
    =
    2R\omega_K.
  \end{align*}
  Therefore,
  \begin{align*}
    \theta_{ij}\leq2\omega_K.
  \end{align*}
  Because
  \begin{align*}
    0<
    \frac{\theta_{ij}}2
    \leq
    \omega_K
    <
    \frac{\pi}{2},
  \end{align*}
  the monotonicity of cosine gives
  \begin{align*}
    \cos\frac{\theta_{ij}}2
    \geq
    \cos\omega_K
    =
    \frac{d_T}{R}
    =
    \chi_K^{-1}.
  \end{align*}
\end{proof}

\begin{lemma}[Exact spherical--chordal angle relation]
  \label{lem:exact-angle-relation}
  For mutually distinct $i,j,k\in\{1,2,3\}$,
  \begin{align}
    \cos\beta_i
    &=
    \cos\frac{\theta_{ij}}2
    \cos\frac{\theta_{ik}}2\cos\alpha_i
    +
    \sin\frac{\theta_{ij}}2
    \sin\frac{\theta_{ik}}2,
    \label{eq:angle-cosine-relation}
  \end{align}
  and
  \begin{align}
    \frac{d_T}{R}\sin\beta_i
    &=
    \cos\frac{\theta_{ij}}2
    \cos\frac{\theta_{ik}}2
    \sin\alpha_i.
    \label{eq:angle-sine-relation}
  \end{align}
\end{lemma}

\begin{proof}
  We set
  \begin{align*}
    n_i:=\frac{p_i}{R}.
  \end{align*}
  From \eqref{eq:vertex-tangent-direction},
  \begin{align*}
    \frac{p_j}{R}
    =
    \cos\theta_{ij}\,n_i
    +
    \sin\theta_{ij}\,t_{ij}.
  \end{align*}
  Therefore,
  \begin{align*}
    p_j-p_i
    =
    R\left(
      (\cos\theta_{ij}-1)n_i
      +
      \sin\theta_{ij}\,t_{ij}
    \right).
  \end{align*}
  Because
  \begin{align*}
    |p_j-p_i|
    =
    2R\sin\frac{\theta_{ij}}2,
  \end{align*}
  the half-angle identities give
  \begin{align}
    \frac{p_j-p_i}{|p_j-p_i|}
    =
    \cos\frac{\theta_{ij}}2\,t_{ij}
    -
    \sin\frac{\theta_{ij}}2\,n_i.
    \label{eq:chord-direction-decomposition}
  \end{align}
  Taking the scalar product of \eqref{eq:chord-direction-decomposition} with its analogue for $p_k$ and using
  \begin{align*}
    t_{ij}\cdot n_i
    =
    t_{ik}\cdot n_i
    =
    0,
    \quad
    t_{ij}\cdot t_{ik}
    =
    \cos\alpha_i,
  \end{align*}
  proves \eqref{eq:angle-cosine-relation}.

  We set
  \begin{align*}
    u_{ij}
    :=
    \frac{p_j-p_i}{|p_j-p_i|},
    \quad
    u_{ik}
    :=
    \frac{p_k-p_i}{|p_k-p_i|}.
  \end{align*}
  Taking the scalar product with $n_i$ after crossing the two decompositions in \eqref{eq:chord-direction-decomposition} yields
  \begin{align*}
    n_i\cdot(u_{ij}\times u_{ik})
    =
    \cos\frac{\theta_{ij}}2
    \cos\frac{\theta_{ik}}2
    n_i\cdot(t_{ij}\times t_{ik}).
  \end{align*}
  Because $t_{ij}$ and $t_{ik}$ are unit tangent vectors forming the angle $\alpha_i$,
  \begin{align}
    \left|
      n_i\cdot(u_{ij}\times u_{ik})
    \right|
    =
    \cos\frac{\theta_{ij}}2
    \cos\frac{\theta_{ik}}2
    \sin\alpha_i.
    \label{eq:triple-product-identity}
  \end{align}
On the other hand, $u_{ij},u_{ik}\in V_T$. Therefore, their cross product is parallel to $\nu_T$, and because they are unit vectors forming the chordal angle $\beta_i$,
  \begin{align*}
    u_{ij}\times u_{ik}
    =
    \pm\sin\beta_i\,\nu_T.
  \end{align*}
  Furthermore,
  \begin{align*}
    n_i\cdot\nu_T
    =
    \frac{d_T}{R}>0.
  \end{align*}
  Therefore,
  \begin{align*}
    \left|
      n_i\cdot(u_{ij}\times u_{ik})
    \right|
    =
    \frac{d_T}{R}\sin\beta_i.
  \end{align*}
  Comparing this with \eqref{eq:triple-product-identity} proves \eqref{eq:angle-sine-relation}.
\end{proof}

\begin{theorem}[Quantitative comparison of chordal and spherical angle factors]
  \label{thm:angle-comparison}
  For $i=1,2,3$,
  \begin{align}
    \chi_K^{-1}\sin\beta_i
    \leq
    \sin\alpha_i
    \leq
    \chi_K\sin\beta_i.
    \label{eq:angle-two-sided-comparison}
  \end{align}
\end{theorem}

\begin{proof}
  Let $j$ and $k$ be the two indices different from $i$. From \eqref{eq:angle-sine-relation},
  \begin{align*}
    \sin\alpha_i
    =
    \chi_K^{-1}
    \frac{
      \sin\beta_i
    }{
      \cos(\theta_{ij}/2)\cos(\theta_{ik}/2)
    }.
  \end{align*}
  Both half-angle cosines are positive and their product is at most one. Therefore,
  \begin{align*}
    \sin\alpha_i
    \geq
    \chi_K^{-1}\sin\beta_i.
  \end{align*}
  From Lemma~\ref{lem:side-localisation},
  \begin{align*}
    \cos\frac{\theta_{ij}}2
    \cos\frac{\theta_{ik}}2
    \geq
    \chi_K^{-2}.
  \end{align*}
  Therefore,
  \begin{align*}
    \sin\alpha_i
    \leq
    \chi_K\sin\beta_i.
  \end{align*}
\end{proof}

\begin{theorem}[Exact comparison of circumradius and spherical semi-regularity parameters]
  \label{thm:exact-comparison}
  Set
  \begin{align*}
    a_T:=\frac{h_{T,1}}{2R},
    \quad
    b_T:=\frac{h_{T,2}}{2R},
    \quad
    0<b_T\leq a_T<1.
  \end{align*}
  Then,
  \begin{align}
    \frac{\mu_K^{\mathrm{circ}}}{\mu_K^{\mathrm{sr}}}
    =
    \frac{\mu_T}{\mu_K^{\mathrm{sr}}}
    =
    \frac{\sin\alpha_K}{\sin\beta_T}
    =
    \frac{
      1
    }{
      \chi_K
      \sqrt{(1-a_T^2)(1-b_T^2)}
    }.
    \label{eq:exact-comparison}
  \end{align}
  Equivalently, denoting by $\beta_2,\beta_3$ the two remaining chordal interior angles of $T$,
  \begin{align}
    \frac{\mu_K^{\mathrm{circ}}}{\mu_K^{\mathrm{sr}}}
    =
    \sqrt{
      \frac{
        1-\vartheta_K^2
      }{
        (1-\vartheta_K^2\sin^2\beta_2)
        (1-\vartheta_K^2\sin^2\beta_3)
      }
    }.
    \label{eq:exact-comparison-angles}
  \end{align}
\end{theorem}

\begin{proof}
  From \eqref{eq:canonical-largest-angles},
  \begin{align*}
    \beta_1=\beta_T,
    \quad
    \alpha_1=\alpha_K.
  \end{align*}
  Therefore, \eqref{eq:angle-sine-relation} with $i=1$ gives
  \begin{align*}
    \frac{\sin\alpha_K}{\sin\beta_T}
    =
    \frac{
      d_T/R
    }{
      \cos(\theta_{12}/2)
      \cos(\theta_{13}/2)
    }.
  \end{align*}
  Because
  \begin{align*}
    a_T
    =
    \frac{h_{T,1}}{2R}
    =
    \sin\frac{\theta_{12}}2,
    \quad
    b_T
    =
    \frac{h_{T,2}}{2R}
    =
    \sin\frac{\theta_{13}}2,
  \end{align*}
  and $\theta_{12},\theta_{13}\in(0,\pi)$,
  \begin{align*}
    \cos\frac{\theta_{12}}2
    =
    \sqrt{1-a_T^2},
    \quad
    \cos\frac{\theta_{13}}2
    =
    \sqrt{1-b_T^2}.
  \end{align*}
  Using $d_T/R=\chi_K^{-1}$ gives
  \begin{align*}
    \frac{\sin\alpha_K}{\sin\beta_T}
    =
    \frac{
      1
    }{
      \chi_K\sqrt{(1-a_T^2)(1-b_T^2)}
    }.
  \end{align*}
  Because
  \begin{align*}
    \frac{\mu_T}{\mu_K^{\mathrm{sr}}}
    =
    \frac{\sin\alpha_K}{\sin\beta_T},
    \quad
    \mu_K^{\mathrm{circ}}
    =
    \mu_T,
  \end{align*}
  this proves \eqref{eq:exact-comparison}.

  From the Euclidean sine law,
  \begin{align*}
    h_{T,1}
    =
    2\mathcal R_T\sin\beta_3,
    \quad
    h_{T,2}
    =
    2\mathcal R_T\sin\beta_2.
  \end{align*}
  Because
  $\vartheta_K=\mathcal R_T/R$,
  \begin{align*}
    a_T
    =
    \vartheta_K\sin\beta_3,
    \quad
    b_T
    =
    \vartheta_K\sin\beta_2.
  \end{align*}
  Substitution into \eqref{eq:exact-comparison}, together with $\chi_K^{-2}=1-\vartheta_K^2$, yields \eqref{eq:exact-comparison-angles}.
\end{proof}

\subsection{Sharp comparison bounds and spherical locality}
\label{subsec:sharp-comparisons}
The exact comparison above yields a locality-independent estimate with a sharp universal constant.

\begin{theorem}[Sharp global comparison of the circumradius and spherical semi-regularity parameters]
  \label{thm:global-semiregularity-comparison}
  \label{thm:sharp-spherical-chordal-bridge}
  For any non-degenerate exact spherical triangle,
  \begin{align}
    \mu_K^{\mathrm{circ}}
    \leq
    \frac{2}{\sqrt3}\,
    \mu_K^{\mathrm{sr}}.
    \label{eq:global-semiregularity-comparison}
  \end{align}
  The constant $2/\sqrt3$ is sharp. Equality holds if and only if the chordal affine core $T$ is equilateral and
  \begin{align*}
    \vartheta_K^2=\frac23.
  \end{align*}
\end{theorem}

\begin{proof}
  We set
  \begin{align*}
    s:=\vartheta_K^2\in(0,1),
    \quad
    a:=\beta_2,
    \quad
    b:=\beta_3,
    \quad
    A:=\sin^2a,
    \quad
    B:=\sin^2b.
  \end{align*}
  From \eqref{eq:exact-comparison-angles},
  \begin{align}
    \left(
      \frac{\mu_K^{\mathrm{circ}}}{\mu_K^{\mathrm{sr}}}
    \right)^2
    =
    \frac{1-s}{(1-As)(1-Bs)}.
    \label{eq:global-comparison-F}
  \end{align}
  For fixed $a$ and $b$, we define
  \begin{align*}
    F(s)
    :=
    \frac{1-s}{(1-As)(1-Bs)},
    \quad
    0\leq s<1,
  \end{align*}
  where the value $F(0)=1$ is the continuous extension of the expression corresponding to the geometric range $s\in(0,1)$.

  Because $\beta_T=\beta_1$ is the largest interior angle of the planar triangle $T$, one has
  \begin{align*}
    a,b\in(0,\pi/2).
  \end{align*}
  Indeed, if $\beta_T<\pi/2$, then $a,b\leq\beta_T<\pi/2$, whereas if $\beta_T\geq\pi/2$, then
  \begin{align*}
    a+b
    =
    \pi-\beta_T
    \leq
    \frac{\pi}{2}.
  \end{align*}
  We set
  \begin{align*}
    \sigma
    :=
    a+b
    =
    \pi-\beta_T.
  \end{align*}
  Because $\beta_T\geq\pi/3$,
  \begin{align*}
    0<\sigma\leq\frac{2\pi}{3}.
  \end{align*}

  Because $0<A,B<1$ and $0<s<1$, the denominator in \eqref{eq:global-comparison-F} is positive. Differentiation gives
  \begin{align*}
    F'(s)
    =
    \frac{
      N(s)
    }{
      (1-As)^2(1-Bs)^2
    },
  \end{align*}
  where
  \begin{align}
    N(s)
    :=
    A+B-1-AB(2s-s^2).
    \label{eq:global-comparison-N}
  \end{align}
  Thus, the sign of $F'(s)$ is the sign of $N(s)$.

  Suppose that
  \begin{align*}
    0<\sigma\leq\frac{\pi}{2}.
  \end{align*}
  Using
  \begin{align*}
    \sin^2a+\sin^2b
    =
    1-\cos(a+b)\cos(a-b),
  \end{align*}
  we have
  \begin{align*}
    A+B
    =
    1-\cos\sigma\cos(a-b)
    \leq
    1,
  \end{align*}
  because $\cos\sigma\geq0$ and $\cos(a-b)>0$. Because $AB>0$ and
  \begin{align*}
    2s-s^2
    =
    s(2-s)>0,
    \quad
    0<s<1,
  \end{align*}
  it follows that
  \begin{align*}
    N(s)<0,
    \quad
    0<s<1.
  \end{align*}
  Therefore, $F$ is strictly decreasing on $(0,1)$, and therefore
  \begin{align}
    F(s)
    <
    F(0)
    =
    1.
    \label{eq:global-comparison-first-case}
  \end{align}
  Suppose that
  \begin{align*}
    \frac{\pi}{2}
    <
    \sigma
    \leq
    \frac{2\pi}{3}.
  \end{align*}
  Because $\cos\sigma<0$ and $\cos(a-b)>0$,
  \begin{align*}
    A+B
    =
    1-\cos\sigma\cos(a-b)
    >
    1.
  \end{align*}
  Furthermore,
  \begin{align*}
    N'(s)
    =
    -2AB(1-s)
    <
    0,
    \quad
    0<s<1,
  \end{align*}
  so $N$ is strictly decreasing. 
  
  We set
  \begin{align*}
    q:=\cot a\cot b.
  \end{align*}
  A direct calculation gives
  \begin{align}
    \frac{A+B-1}{AB}
    =
    1-\cot^2a\cot^2b
    =
    1-q^2.
    \label{eq:global-comparison-q}
  \end{align}
  Because $A+B>1$, the left-hand side is positive, and hence
  \begin{align*}
    0<q<1.
  \end{align*}
  We define
  \begin{align}
    s_*
    :=
    1-q
    =
    1-\cot a\cot b.
    \label{eq:global-comparison-sstar}
  \end{align}
  Then, $s_*\in(0,1)$ and
  \begin{align*}
    2s_*-s_*^2
    =
    1-q^2
    =
    \frac{A+B-1}{AB}.
  \end{align*}
  Therefore,
  \begin{align*}
    N(s_*)=0.
  \end{align*}
  Because $N$ is strictly decreasing, $s_*$ is its unique zero on $(0,1)$. Consequently, $F$ is strictly increasing on $(0,s_*)$ and strictly decreasing on $(s_*,1)$, and therefore attains its unique maximum at $s_*$. At this point,
  \begin{align*}
    1-s_*
    =
    \cot a\cot b.
  \end{align*}
  Furthermore,
  \begin{align*}
    1-s_*\sin^2a
    &=
    1-
    \left(
      1-\cot a\cot b
    \right)
    \sin^2a
    \\
    &=
    \cos^2a
    +
    \sin^2a\cot a\cot b
    \\
    &=
    \frac{
      \cos a\,\sin(a+b)
    }{
      \sin b
    }.
  \end{align*}
  Similarly,
  \begin{align*}
    1-s_*\sin^2b
    =
    \frac{
      \cos b\,\sin(a+b)
    }{
      \sin a
    }.
  \end{align*}
  Substitution into \eqref{eq:global-comparison-F} gives
  \begin{align*}
    F(s_*)
    &=
    \frac{
      \cot a\cot b
    }{
      \displaystyle
      \frac{\cos a\,\sin(a+b)}{\sin b}
      \frac{\cos b\,\sin(a+b)}{\sin a}
    }
    =
    \frac{1}{\sin^2(a+b)}
    =
    \frac{1}{\sin^2\sigma}.
  \end{align*}
  Because
  \begin{align*}
    \frac{\pi}{2}
    <
    \sigma
    \leq
    \frac{2\pi}{3},
  \end{align*}
  and $\sin$ is decreasing on this interval,
  \begin{align*}
    \sin\sigma
    \geq
    \sin\frac{2\pi}{3}
    =
    \frac{\sqrt3}{2}.
  \end{align*}
  Therefore,
  \begin{align}
    F(s)
    \leq
    F(s_*)
    =
    \frac{1}{\sin^2\sigma}
    \leq
    \frac43.
    \label{eq:global-comparison-second-case}
  \end{align}

  Combining \eqref{eq:global-comparison-first-case} and \eqref{eq:global-comparison-second-case} gives
  \begin{align*}
    F(s)\leq\frac43.
  \end{align*}
  Because all parameters are positive,
  \begin{align*}
    \frac{\mu_K^{\mathrm{circ}}}{\mu_K^{\mathrm{sr}}}
    \leq
    \frac{2}{\sqrt3}.
  \end{align*}

  It remains to characterise equality. The first case gives $F(s)<1$, so equality in the final estimate can occur only in the second case, and only if
  \begin{align*}
    s=s_*,
    \quad
    \frac{1}{\sin^2\sigma}
    =
    \frac43.
  \end{align*}
  Because $\sigma\in(\pi/2,2\pi/3]$, the latter condition is equivalent to
  \begin{align*}
    \sigma
    =
    \frac{2\pi}{3}.
  \end{align*}
  Therefore,
  \begin{align*}
    \beta_T
    =
    \pi-\sigma
    =
    \frac{\pi}{3}.
  \end{align*}
  Because $\beta_T$ is the largest chordal interior angle, all three interior angles of $T$ must equal $\pi/3$. Thus, $T$ is equilateral and
  \begin{align*}
    a=b=\frac{\pi}{3}.
  \end{align*}
  Consequently,
  \begin{align*}
    s_*
    =
    1-\cot^2\frac{\pi}{3}
    =
    \frac23.
  \end{align*}
  Because $s=\vartheta_K^2$, equality implies
  \begin{align*}
    \vartheta_K^2=\frac23.
  \end{align*}
Conversely, if $T$ is equilateral and $\vartheta_K^2=2/3$, then
  \begin{align*}
    a=b=\frac{\pi}{3},
    \quad
    A=B=\frac34,
    \quad
    s=\frac23.
  \end{align*}
  Therefore,
  \begin{align*}
    F\left(\frac23\right)
    &=
    \frac{
      1-\frac23
    }{
      \left(
        1-\frac34\frac23
      \right)^2
    }
    =
    \frac{
      \frac13
    }{
      \left(\frac12\right)^2
    }
    =
    \frac43.
  \end{align*}
  Thus,
  \begin{align*}
    \frac{\mu_K^{\mathrm{circ}}}{\mu_K^{\mathrm{sr}}}
    =
    \frac{2}{\sqrt3},
  \end{align*}
  and equality is attained.
\end{proof}

\begin{corollary}[Spherical semi-regularity implies chordal semi-regularity]
  \label{cor:spherical-implies-chordal}
  Let $\mathscr K$ be any family of non-degenerate exact spherical triangles. If
  \begin{align*}
    \mu_K^{\mathrm{sr}}
    \leq
    \gamma_{\mathrm{sr}}
    \quad
    \forall K\in\mathscr K,
  \end{align*}
  then the associated chordal affine cores satisfy
  \begin{align*}
    \mu_T
    =
    \mu_K^{\mathrm{circ}}
    \leq
    \frac{2}{\sqrt3}\,
    \gamma_{\mathrm{sr}}
    \quad
    \forall K\in\mathscr K.
  \end{align*}
  In particular, the intrinsic spherical maximum-angle condition implies the chordal maximum-angle condition without any spherical-locality assumption.
\end{corollary}

\begin{proof}
  This follows immediately from Theorem~\ref{thm:global-semiregularity-comparison} and $\mu_K^{\mathrm{circ}}=\mu_T$.
\end{proof}

\begin{example}[Failure of the reverse implication without spherical locality]
  \label{ex:coarse-equilateral-spherical}
  We consider the coarse equilateral family of Example~\ref{ex:coarse-equilateral} and set
  \begin{align*}
    t:=\frac{d_T}{R}\in(0,1).
  \end{align*}
  The three geodesic sides are equal, and their common central angle $\theta_K$ satisfies
  \begin{align*}
    \cos\theta_K
    =
    \frac{3t^2-1}{2}.
  \end{align*}
  By the spherical law of cosines, the three spherical interior angles are also equal, and their common value $\alpha_K$ satisfies
  \begin{align*}
    \cos\alpha_K
    =
    \frac{3t^2-1}{1+3t^2},
    \quad
    \sin\alpha_K
    =
    \frac{2\sqrt3\,t}{1+3t^2}.
  \end{align*}
  Therefore,
  \begin{align*}
    \alpha_K
    \longrightarrow
    \pi
    \quad
    (t\downarrow0),
  \end{align*}
  and Theorem~\ref{thm:spherical-sr-maximum-angle-identity} gives
  \begin{align*}
    \mu_K^{\mathrm{sr}}
    =
    \frac{2}{\sin\alpha_K}
    =
    \frac{1+3t^2}{\sqrt3\,t}
    \longrightarrow
    \infty.
  \end{align*}
  On the other hand,
  \begin{align*}
    \mu_K^{\mathrm{circ}}
    =
    \mu_T
    =
    \frac{4}{\sqrt3}.
  \end{align*}
  Furthermore,
  \begin{align*}
    \vartheta_K
    =
    \sqrt{1-t^2}
    \longrightarrow1,
    \quad
    \chi_K
    =
    \frac1t
    \longrightarrow\infty.
  \end{align*}
  Therefore,
  \begin{align*}
    \frac{\mu_K^{\mathrm{sr}}}{\mu_K^{\mathrm{circ}}}
    \longrightarrow\infty.
  \end{align*}
  Thus, uniform chordal semi-regularity does not imply intrinsic spherical semi-regularity without spherical locality. In particular, no uniform reverse estimate
  \begin{align*}
    \mu_K^{\mathrm{sr}}
    \leq
    C\mu_K^{\mathrm{circ}}
  \end{align*}
  can hold over arbitrary exact spherical triangles.
\end{example}

\begin{theorem}[Sharp unit-prefactor power comparison]
  \label{thm:sharp-semiregularity-comparison}
  For any non-degenerate exact spherical triangle,
  \begin{align}
    \chi_K^{-1}\mu_K^{\mathrm{sr}}
    <
    \mu_K^{\mathrm{circ}}
    <
    \chi_K^{1/2}\mu_K^{\mathrm{sr}}.
    \label{eq:sharp-semiregularity-comparison}
  \end{align}
  Among upper bounds with unit prefactor of the form
  \begin{align*}
    \frac{\mu_K^{\mathrm{circ}}}{\mu_K^{\mathrm{sr}}}
    \leq
    \chi_K^\tau,
  \end{align*}
  the exponent $1/2$ is optimal: no exponent $\tau<1/2$ works uniformly over all exact spherical triangles. The lower bound is asymptotically sharp: along families with $a_T,b_T\to0$,
  \begin{align*}
    \frac{
      \mu_K^{\mathrm{circ}}/\mu_K^{\mathrm{sr}}
    }{
      \chi_K^{-1}
    }
    \longrightarrow1.
  \end{align*}
\end{theorem}

\begin{proof}
  Because $\mu_K^{\mathrm{circ}}=\mu_T$ and $\chi_K=\chi_T$, we may use \eqref{eq:exact-comparison}. Since $0<a_T,b_T<1$,
  \begin{align*}
    0<(1-a_T^2)(1-b_T^2)<1.
  \end{align*}
  Therefore,
  \begin{align*}
    \frac{\mu_K^{\mathrm{circ}}}{\mu_K^{\mathrm{sr}}}
    >
    \chi_K^{-1}.
  \end{align*}
  Furthermore, if $a_T,b_T\to0$, then
  \begin{align*}
    \frac{
      \mu_K^{\mathrm{circ}}/\mu_K^{\mathrm{sr}}
    }{
      \chi_K^{-1}
    }
    =
    \frac{
      1
    }{
      \sqrt{(1-a_T^2)(1-b_T^2)}
    }
    \longrightarrow1,
  \end{align*}
  proving the asserted asymptotic sharpness of the lower bound.

  For the upper bound, we put
  \begin{align*}
    s:=\vartheta_K^2\in(0,1),
    \quad
    x:=\sin^2\beta_2,
    \quad
    y:=\sin^2\beta_3,
  \end{align*}
  and define
  \begin{align*}
    \varphi(\xi)
    :=
    \frac{
      \log(1-s\xi)
    }{
      \log(1-s)
    },
    \quad
    \xi\in[0,1].
  \end{align*}
  Then,
  \begin{align*}
    \varphi(0)=0,
    \quad
    \varphi(1)=1.
  \end{align*}
  Because $\xi\mapsto\log(1-s\xi)$ is strictly concave and $\log(1-s)<0$, the function $\varphi$ is strictly convex. Therefore,
  \begin{align*}
    \varphi(\xi)\leq\xi,
    \quad
    0\leq\xi\leq1,
  \end{align*}
  with strict inequality for $\xi\in(0,1)$. Taking logarithms in \eqref{eq:exact-comparison-angles} and dividing by
  \begin{align*}
    \log\chi_K
    =
    -\frac12\log(1-s)
    >
    0
  \end{align*}
  gives
  \begin{align}
    \frac{
      \log(
        \mu_K^{\mathrm{circ}}/\mu_K^{\mathrm{sr}}
      )
    }{
      \log\chi_K
    }
    =
    -1+\varphi(x)+\varphi(y)
    \leq
    -1+\sin^2\beta_2+\sin^2\beta_3.
    \label{eq:exponent-bound}
  \end{align}
  We write
  \begin{align*}
    \sigma
    :=
    \beta_2+\beta_3
    =
    \pi-\beta_T,
    \quad
    \delta
    :=
    |\beta_2-\beta_3|.
  \end{align*}
  Then,
  \begin{align*}
    \sin^2\beta_2+\sin^2\beta_3
    =
    1-\cos\sigma\cos\delta.
  \end{align*}
  Because $\beta_T$ is the largest interior angle, $\beta_T\geq\pi/3$, and therefore
  \begin{align*}
    \sigma\leq\frac{2\pi}{3}.
  \end{align*}

  Suppose first that
  \begin{align*}
    \sigma\geq\frac{\pi}{2}.
  \end{align*}
  Because $\beta_2,\beta_3\in(0,\pi/2)$, one has $0\leq\delta<\pi/2$, and hence
  \begin{align*}
    0<\cos\delta\leq1.
  \end{align*}
  Since $\cos\sigma\leq0$,
  \begin{align*}
    \cos\sigma\cos\delta
    \geq
    \cos\sigma.
  \end{align*}
  Consequently,
  \begin{align*}
    1-\cos\sigma\cos\delta
    \leq
    1-\cos\sigma
    \leq
    1-\cos\frac{2\pi}{3}
    =
    \frac32.
  \end{align*}

  Suppose next that
  \begin{align*}
    \sigma<\frac{\pi}{2}.
  \end{align*}
  Because $\beta_2,\beta_3>0$,
  \begin{align*}
    \delta
    =
    |\beta_2-\beta_3|
    <
    \beta_2+\beta_3
    =
    \sigma.
  \end{align*}
  Therefore,
  \begin{align*}
    \cos\delta
    >
    \cos\sigma
    >
    0,
  \end{align*}
  and hence
  \begin{align*}
    1-\cos\sigma\cos\delta
    <
    1-\cos^2\sigma
    =
    \sin^2\sigma
    <
    1.
  \end{align*}
  
  Thus, in all cases,
  \begin{align*}
    \sin^2\beta_2+\sin^2\beta_3
    \leq
    \frac32.
  \end{align*}
  Equation \eqref{eq:exponent-bound} therefore gives
  \begin{align*}
    \frac{
      \log(
        \mu_K^{\mathrm{circ}}/\mu_K^{\mathrm{sr}}
      )
    }{
      \log\chi_K
    }
    \leq
    \frac12,
  \end{align*}
  and hence
  \begin{align*}
    \mu_K^{\mathrm{circ}}
    \leq
    \chi_K^{1/2}\mu_K^{\mathrm{sr}}.
  \end{align*}

  Because $\beta_2,\beta_3\in(0,\pi/2)$, one has $x,y\in(0,1)$. Thus,
  $\varphi(x)<x$ and $\varphi(y)<y$, so the preceding upper bound is in
  fact strict:
  \begin{align*}
    \mu_K^{\mathrm{circ}}
    <
    \chi_K^{1/2}\mu_K^{\mathrm{sr}}.
  \end{align*}
  To prove optimality within the unit-prefactor power class, we consider the equilateral family
  \begin{align*}
    \beta_1
    =
    \beta_2
    =
    \beta_3
    =
    \frac{\pi}{3}.
  \end{align*}
  Then, $x=y=3/4$, and
  \begin{align*}
    \frac{
      \log(
        \mu_K^{\mathrm{circ}}/\mu_K^{\mathrm{sr}}
      )
    }{
      \log\chi_K
    }
    &=
    -1
    +
    2
    \frac{
      \log(
        1-\tfrac34\vartheta_K^2
      )
    }{
      \log(
        1-\vartheta_K^2
      )
    }
    \longrightarrow
    \frac12
    \quad
    (\vartheta_K\downarrow0).
  \end{align*}
  Therefore, for any $\tau<1/2$, the unit-prefactor inequality
  \begin{align*}
    \frac{\mu_K^{\mathrm{circ}}}{\mu_K^{\mathrm{sr}}}
    \leq
    \chi_K^\tau
  \end{align*}
  fails along this family for all sufficiently small positive $\vartheta_K$.
\end{proof}

\begin{corollary}[Equivalence of the chordal bridge under spherical locality]
  \label{cor:equivalence}
  Let $\mathscr K$ be a uniformly spherically local family of exact spherical triangles with locality constant $\eta_0\in(0,1)$, and set
  \begin{align*}
    \chi_0
    :=
    (1-\eta_0^2)^{-1/2}.
  \end{align*}
  Then, the auxiliary chordal-bridge condition
  \begin{align}
    \mu_K^{\mathrm{circ}}
    =
    \mu_T
    \leq
    \gamma_0
    \quad
    \forall K\in\mathscr K
    \label{eq:chordal-semiregularity-bound}
  \end{align}
  for some $\gamma_0<\infty$ is equivalent to uniform intrinsic spherical semi-regularity and hence to the spherical maximum-angle condition \eqref{eq:spherical-maximum-angle-condition}. More precisely, the chordal bound with constant $\gamma_0$ implies the spherical maximum-angle condition with
  \begin{align*}
    \delta
    :=
    \arcsin\left(
      \frac{2}{\chi_0\gamma_0}
    \right),
  \end{align*}
  whereas the spherical maximum-angle condition with constant $\delta$ implies, without using spherical locality,
  \begin{align*}
    \mu_T
    \leq
    \frac{4}{\sqrt3\,c_\delta},
    \quad
    c_\delta
    :=
    \min\left\{
      \frac{\sqrt3}{2},
      \sin\delta
    \right\}.
  \end{align*}
\end{corollary}

\begin{proof}
  Uniform spherical locality gives
  \begin{align*}
    \chi_K
    \leq
    \chi_0
    \quad
    \forall K\in\mathscr K.
  \end{align*}
  Theorem~\ref{thm:angle-comparison}, together with $\alpha_1=\alpha_K$ and $\beta_1=\beta_T$, gives
  \begin{align}
    \chi_0^{-1}\mu_K^{\mathrm{sr}}
    \leq
    \mu_T
    \leq
    \chi_0\mu_K^{\mathrm{sr}}
    \quad
    \forall K\in\mathscr K.
    \label{eq:uniform-chordal-spherical-parameter-comparison}
  \end{align}
  This proves the qualitative equivalence of uniform chordal and spherical semi-regularity.

  Suppose first that \eqref{eq:chordal-semiregularity-bound} holds. Because $\mu_T\geq2$, necessarily $\gamma_0\geq2$. The left-hand inequality in \eqref{eq:uniform-chordal-spherical-parameter-comparison} yields
  \begin{align*}
    \mu_K^{\mathrm{sr}}
    \leq
    \chi_0\gamma_0.
  \end{align*}
  Therefore,
  \begin{align*}
    \sin\alpha_K
    =
    \frac{2}{\mu_K^{\mathrm{sr}}}
    \geq
    \frac{2}{\chi_0\gamma_0}.
  \end{align*}
  Because $0<\alpha_K<\pi$,
  \begin{align*}
    \alpha_K
    \leq
    \pi-
    \arcsin\left(
      \frac{2}{\chi_0\gamma_0}
    \right).
  \end{align*}
  Thus, the spherical maximum-angle condition holds with the stated constant $\delta$. Conversely, suppose
  \begin{align*}
    \alpha_K
    \leq
    \pi-\delta
    \quad
    \forall K\in\mathscr K.
  \end{align*}
  From \eqref{eq:largest-spherical-angle-range},
  \begin{align*}
    \frac{\pi}{3}
    <
    \alpha_K
    \leq
    \pi-\delta.
  \end{align*}
  Therefore,
  \begin{align*}
    \sin\alpha_K
    \geq
    c_\delta
    :=
    \min\left\{
      \frac{\sqrt3}{2},
      \sin\delta
    \right\}
    >0.
  \end{align*}
  Thus,
  \begin{align*}
    \mu_K^{\mathrm{sr}}
    =
    \frac{2}{\sin\alpha_K}
    \leq
    \frac{2}{c_\delta}.
  \end{align*}
  Corollary~\ref{cor:spherical-implies-chordal} then gives
  \begin{align*}
    \mu_T
    \leq
    \frac{2}{\sqrt3}\mu_K^{\mathrm{sr}}
    \leq
    \frac{4}{\sqrt3\,c_\delta}.
  \end{align*}
  This converse implication does not require spherical locality.
\end{proof}

\subsection{The area-based intrinsic parameter}
\label{subsec:area-parameter}
Under Convention~\ref{conv:chordal-labelling}, the two geodesic edges adjacent to the largest spherical angle $\alpha_K$ have lengths $R\theta_{12}$ and $R\theta_{13}$. This motivates the following intrinsic dimensionless quantity.

\begin{definition}[Area-based spherical parameter]
  \label{def:spherical-area-parameter}
  We define
  \begin{align}
    \mu_K^{\mathrm{area}}
    :=
    \frac{
      R^2\theta_{12}\theta_{13}
    }{
      |K|_2
    }.
    \label{eq:spherical-area-parameter}
  \end{align}
\end{definition}

If more than one longest side exists, the value of $\mu_K^{\mathrm{area}}$ is independent of the admissible choice of the opposite vertex $p_1$. Indeed, if two longest sides have equal length, the products of the two geodesic side lengths adjacent to the corresponding opposite vertices coincide; the equilateral case is immediate.

\begin{lemma}[Elementary area bridge]
  \label{lem:spherical-area-parameter}
  For every non-degenerate exact spherical triangle,
  \begin{align}
    \chi_K^{-2}\mu_K^{\mathrm{circ}}
    <
    \mu_K^{\mathrm{area}}.
    \label{eq:area-elementary-lower-bridge}
  \end{align}
\end{lemma}

\begin{proof}
  For $m=2,3$, we set
  \begin{align*}
    x:=\frac{\theta_{1m}}{2}\in\left(0,\frac{\pi}{2}\right).
  \end{align*}
  Because $\sin x<x$ for $x>0$,
  \begin{align*}
    |p_1-p_m|
    =
    2R\sin x
    <
    R\theta_{1m}.
  \end{align*}
  Therefore, the product of the two geodesic edge lengths adjacent to $p_1$ is strictly larger than $h_{T,1}h_{T,2}$. On the other hand, \eqref{eq:radial-jacobian-bounds} gives
  \begin{align*}
    |K|_2
    \leq
    \chi_T^2|T|_2.
  \end{align*}
  Therefore,
  \begin{align*}
    \mu_K^{\mathrm{area}}
    &>
    \frac{
      h_{T,1}h_{T,2}
    }{
      \chi_T^2|T|_2
    }
    =
    \chi_T^{-2}\mu_T
    =
    \chi_K^{-2}\mu_K^{\mathrm{circ}}.
  \end{align*}
\end{proof}

\begin{proposition}[Area--angle comparison]
  \label{prop:area-below-angle}
  For any non-degenerate exact spherical triangle,
  \begin{align}
    \mu_K^{\mathrm{area}}
    <
    \mu_K^{\mathrm{sr}},
    \label{eq:area-below-angle}
  \end{align}
  equivalently,
  \begin{align}
    |K|_2
    >
    \frac12
    R^2\theta_{12}\theta_{13}\sin\alpha_K.
    \label{eq:spherical-area-angle-lower}
  \end{align}
  The constant one in \eqref{eq:area-below-angle} is sharp: if $\theta_{12},\theta_{13}\to0$, then
  \begin{align}
    \frac{
      \mu_K^{\mathrm{area}}
    }{
      \mu_K^{\mathrm{sr}}
    }
    \longrightarrow1.
    \label{eq:area-angle-flat-limit}
  \end{align}
\end{proposition}

\begin{proof}
  We set
  \begin{align*}
    u:=\theta_{12},
    \quad
    v:=\theta_{13},
    \quad
    x:=\frac u2,
    \quad
    y:=\frac v2,
  \end{align*}
  and, after interchanging $p_2$ and $p_3$ if necessary, assume $x\geq y$. We set
  \begin{align*}
    X:=\tan x,
    \quad
    Y:=\tan y,
    \quad
    c:=\cos\alpha_K,
    \quad
    \mathcal E:=\frac{|K|_2}{R^2}.
  \end{align*}
  A classical spherical-excess formula \cite[Chapter VIII]{Todhunter1886} gives
  \begin{align}
    \tan\frac{\mathcal E}{2}
    =
    \frac{
      XY\sin\alpha_K
    }{
      1+XY\cos\alpha_K
    }.
    \label{eq:spherical-excess-tangent}
  \end{align}
Because $0<\mathcal E<2\pi$ and $0<\alpha_K<\pi$, the point
  \begin{align*}
    1+te^{i\alpha_K}
  \end{align*}
  lies in the open upper half-plane for every $t>0$. Therefore, its argument, chosen continuously from zero along
  \begin{align*}
    t\longmapsto1+te^{i\alpha_K},
    \quad
    0\leq t\leq XY,
  \end{align*}
  lies in $(0,\pi)$. Therefore, \eqref{eq:spherical-excess-tangent} yields the branch identity
  \begin{align*}
    \mathcal E
    =
    2\arg\left(
      1+XYe^{i\alpha_K}
    \right).
  \end{align*}
  The path does not meet the origin, and differentiating its argument gives
  \begin{align}
    \mathcal E
    =
    2\sin\alpha_K\,I(c),
    \quad
    I(c)
    :=
    \int_0^{XY}
    \frac{\diff t}{1+2ct+t^2}.
    \label{eq:spherical-excess-integral}
  \end{align}
  Let $w:=\theta_{23}$. Because $\alpha_K=\alpha_1$ is the largest spherical interior angle, Lemma~\ref{lem:spherical-side-angle-ordering} gives
  \begin{align*}
    w\geq u,
    \quad
    w\geq v.
  \end{align*}
  The spherical cosine law reads
  \begin{align}
    \cos w
    =
    \cos u\cos v
    +
    \sin u\sin v\,c.
    \label{eq:area-proof-cosine-law}
  \end{align}
  Because $w\geq u$ and cosine is decreasing on $(0,\pi)$,
  \begin{align*}
    \cos w\leq\cos u.
  \end{align*}
  Using \eqref{eq:area-proof-cosine-law}, we have
  \begin{align*}
    \sin u\sin v\,c
    &\leq
    \cos u(1-\cos v),
  \end{align*}
  and therefore
  \begin{align*}
    c
    &\leq
    \cot u\,
    \frac{1-\cos v}{\sin v}
    =
    \cot u\tan\frac v2.
  \end{align*}
  Similarly, $w\geq v$ gives
  \begin{align*}
    c
    \leq
    \cot v\tan\frac u2.
  \end{align*}
  Therefore,
  \begin{align}
    c
    \leq
    \min\left\{
      \cot u\tan\frac v2,
      \cot v\tan\frac u2
    \right\}.
    \label{eq:area-proof-c-bound}
  \end{align}
  Because $x\geq y$, one has $X\geq Y$. Using
  \begin{align*}
    \cot(2x)
    =
    \frac{1-X^2}{2X},
    \quad
    \cot(2y)
    =
    \frac{1-Y^2}{2Y},
  \end{align*}
  gives
  \begin{align*}
    \cot u\tan\frac v2
    &=
    \frac12
    \left(
      \frac YX-XY
    \right),
    \quad
    \cot v\tan\frac u2
    =
    \frac12
    \left(
      \frac XY-XY
    \right).
  \end{align*}
  Because $X\geq Y$,
  \begin{align*}
    \frac YX\leq\frac XY,
  \end{align*}
  and therefore
  \begin{align}
    c
    \leq
    \frac12
    \left(
      \frac YX-XY
    \right)
    =:m.
    \label{eq:area-proof-m}
  \end{align}
  For $q>-1$, we define
  \begin{align*}
    I(q)
    :=
    \int_0^{XY}
    \frac{\diff t}{1+2qt+t^2}.
  \end{align*}
  Because $m\geq c>-1$, the denominator is positive for any $q\in[c,m]$ and $t\geq0$. Indeed, this is immediate if $q\geq0$, while for $-1<q<0$,
  \begin{align*}
    1+2qt+t^2
    =
    (t+q)^2+1-q^2
    >
    0.
  \end{align*}
  Furthermore,
  \begin{align*}
    \frac{\partial}{\partial q}
    \frac{1}{1+2qt+t^2}
    =
    -\frac{2t}{(1+2qt+t^2)^2}
    \leq0.
  \end{align*}
  Thus, the integrand is decreasing with respect to $q$, and
  \begin{align*}
    I(c)\geq I(m).
  \end{align*}
  Substituting $t=XYs$ and using \eqref{eq:area-proof-m} gives
  \begin{align*}
    I(m)
    &=
    XY
    \int_0^1
    \frac{
      \diff s
    }{
      1+Y^2s-X^2Y^2s(1-s)
    }
    >
    XY
    \int_0^1
    \frac{
      \diff s
    }{
      1+Y^2s
    }
    =
    \frac XY
    \log(1+Y^2).
  \end{align*}
  The function $\tan t/t$ is increasing on $(0,\pi/2)$, and hence
  \begin{align*}
    \frac XY
    \geq
    \frac xy.
  \end{align*}
  Furthermore,
  \begin{align*}
    \log(1+Y^2)
    =
    2\log\sec y.
  \end{align*}
  The function
  \begin{align*}
    f(y)
    :=
    2\log\sec y-y^2
  \end{align*}
  satisfies $f(0)=0$ and
  \begin{align*}
    f'(y)
    =
    2(\tan y-y)
    >
    0,
    \quad
    y>0.
  \end{align*}
  Therefore,
  \begin{align*}
    \log(1+Y^2)>y^2.
  \end{align*}
  Consequently,
  \begin{align*}
    I(c)
    >
    \frac XY\log(1+Y^2)
    \geq
    \frac xy\log(1+Y^2)
    >
    xy.
  \end{align*}
  Using \eqref{eq:spherical-excess-integral}, we obtain
  \begin{align*}
    \mathcal E
    >
    2xy\sin\alpha_K
    =
    \frac{uv}{2}\sin\alpha_K.
  \end{align*}
  Because
  \begin{align*}
    \mathcal E
    =
    \frac{|K|_2}{R^2},
    \quad
    u=\theta_{12},
    \quad
    v=\theta_{13},
  \end{align*}
  this proves \eqref{eq:spherical-area-angle-lower}, and hence \eqref{eq:area-below-angle}.

  It remains to prove sharpness. Suppose that $u,v\to0$. Then,
  $x,y\to0$ and
  \begin{align*}
    \frac{XY}{xy}
    =
    \frac{\tan x}{x}
    \frac{\tan y}{y}
    \longrightarrow1.
  \end{align*}
  Because $c\in[-1,1]$,
  \begin{align*}
    (1-t)^2
    \leq
    1+2ct+t^2
    \leq
    (1+t)^2.
  \end{align*}
  For all sufficiently small $XY<1$, this holds with positive lower bound for any $0\leq t\leq XY$. Taking reciprocals and integrating gives
  \begin{align*}
    \frac{XY}{1+XY}
    \leq
    I(c)
    \leq
    \frac{XY}{1-XY}.
  \end{align*}
  Therefore,
  \begin{align*}
    \frac{I(c)}{XY}
    \longrightarrow1
  \end{align*}
  uniformly for $c\in[-1,1]$. Because $XY/(xy)\to1$, it follows that
  \begin{align*}
    \frac{I(c)}{xy}
    \longrightarrow1.
  \end{align*}
  Finally, by \eqref{eq:spherical-excess-integral},
  \begin{align*}
    \frac{
      \mu_K^{\mathrm{area}}
    }{
      \mu_K^{\mathrm{sr}}
    }
    =
    \frac{
      R^2uv/|K|_2
    }{
      2/\sin\alpha_K
    }
    =
    \frac{uv\sin\alpha_K}{2\mathcal E}
    =
    \frac{xy}{I(c)}
    \longrightarrow1.
  \end{align*}
\end{proof}

\begin{theorem}[Sharpened area comparison]
  \label{thm:area-sharp-comparison}
  For any non-degenerate exact spherical triangle,
  \begin{align}
    \chi_K^{-2}\mu_K^{\mathrm{circ}}
    <
    \mu_K^{\mathrm{area}}
    <
    \chi_K\mu_K^{\mathrm{circ}},
    \label{eq:area-chordal-sharp-comparison}
  \end{align}
  and
  \begin{align}
    \chi_K^{-3}\mu_K^{\mathrm{sr}}
    <
    \mu_K^{\mathrm{area}}
    <
    \mu_K^{\mathrm{sr}}.
    \label{eq:area-spherical-sharp-comparison}
  \end{align}
  Furthermore, the exponent one in the upper bound of \eqref{eq:area-chordal-sharp-comparison} is optimal. More precisely, for any $\sigma<1$,
  \begin{align}
    \sup_K
    \frac{
      \mu_K^{\mathrm{area}}
    }{
      \chi_K^\sigma
      \mu_K^{\mathrm{circ}}
    }
    =
    \infty,
    \label{eq:area-upper-exponent-optimality}
  \end{align}
  where the supremum is taken over all non-degenerate exact spherical triangles.
\end{theorem}

\begin{proof}
  From \eqref{eq:intrinsic-chordal-bridge},
  \begin{align*}
    \mu_K^{\mathrm{circ}}
    =
    \mu_T,
    \quad
    \chi_K
    =
    \chi_T.
  \end{align*}
  Lemma~\ref{lem:spherical-area-parameter} gives
  \begin{align*}
    \mu_K^{\mathrm{area}}
    >
    \chi_K^{-2}\mu_K^{\mathrm{circ}}.
  \end{align*}
  On the other hand, Proposition~\ref{prop:area-below-angle} gives
  \begin{align*}
    \mu_K^{\mathrm{area}}
    <
    \mu_K^{\mathrm{sr}},
  \end{align*}
  while the strict lower bound in Theorem~\ref{thm:sharp-semiregularity-comparison} gives
  \begin{align*}
    \chi_K^{-1}\mu_K^{\mathrm{sr}}
    <
    \mu_K^{\mathrm{circ}}.
  \end{align*}
  Equivalently,
  \begin{align*}
    \mu_K^{\mathrm{sr}}
    <
    \chi_K\mu_K^{\mathrm{circ}}.
  \end{align*}
  Therefore,
  \begin{align*}
    \mu_K^{\mathrm{area}}
    <
    \mu_K^{\mathrm{sr}}
    <
    \chi_K\mu_K^{\mathrm{circ}},
  \end{align*}
  which proves the upper bound in \eqref{eq:area-chordal-sharp-comparison}.

  Combining
  \begin{align*}
    \mu_K^{\mathrm{area}}
    >
    \chi_K^{-2}\mu_K^{\mathrm{circ}}
  \end{align*}
  with
  \begin{align*}
    \mu_K^{\mathrm{circ}}
    >
    \chi_K^{-1}\mu_K^{\mathrm{sr}}
  \end{align*}
  gives
  \begin{align*}
    \mu_K^{\mathrm{area}}
    >
    \chi_K^{-3}\mu_K^{\mathrm{sr}}.
  \end{align*}
  Together with Proposition~\ref{prop:area-below-angle}, this proves \eqref{eq:area-spherical-sharp-comparison}.

  It remains to prove the optimality of exponent one in the upper bound of \eqref{eq:area-chordal-sharp-comparison}. Fix $\chi>1$ and set
  \begin{align*}
    d
    :=
    \frac{R}{\chi},
    \quad
    \rho
    :=
    R\sqrt{1-\chi^{-2}}.
  \end{align*}
  Then,
  \begin{align*}
    d^2+\rho^2
    =
    R^2.
  \end{align*}
  Therefore, the intersection of $\Sph_R^2$ with a plane at distance $d$ from the origin is a circle of radius $\rho$. For
  \begin{align*}
    0<s<\frac{2\pi}{3},
  \end{align*}
  we choose three points on this circle at polar angles $-s$, $0$, and $s$, and label the point at polar angle $0$ as $p_1$. The two chordal edges issuing from $p_1$ have the common length
  \begin{align*}
    h_{T,1}
    =
    h_{T,2}
    =
    2\rho\sin\frac{s}{2},
  \end{align*}
  while the opposite chordal edge has length
  \begin{align*}
    h_T
    =
    2\rho\sin s.
  \end{align*}
  Because
  \begin{align*}
    \frac{\sin s}{\sin(s/2)}
    =
    2\cos\frac{s}{2}
    >
    1
    \quad
    \left(
      0<s<\frac{2\pi}{3}
    \right),
  \end{align*}
  the opposite edge is the unique longest chordal edge. The corresponding chordal interior angles are
  \begin{align*}
    \beta_T
    =
    \pi-s,
    \quad
    \beta_2
    =
    \beta_3
    =
    \frac{s}{2}.
  \end{align*}
  Thus, the labelling agrees with Convention~\ref{conv:chordal-labelling}. Furthermore,
  \begin{align*}
    \chi_K
    =
    \chi_T
    =
    \frac{R}{d}
    =
    \chi.
  \end{align*}
  The normalised adjacent chord lengths are
  \begin{align*}
    a_T
    =
    b_T
    =
    \frac{\rho}{R}
    \sin\frac{s}{2}
    =
    \sqrt{1-\chi^{-2}}
    \sin\frac{s}{2},
  \end{align*}
  and therefore
  \begin{align*}
    a_T,b_T
    \longrightarrow0
    \quad
    (s\downarrow0).
  \end{align*}
  The exact comparison \eqref{eq:exact-comparison} gives
  \begin{align*}
    \frac{
      \mu_K^{\mathrm{circ}}
    }{
      \mu_K^{\mathrm{sr}}
    }
    &=
    \frac{
      1
    }{
      \chi
      \sqrt{
        (1-a_T^2)(1-b_T^2)
      }
    }
    \longrightarrow
    \chi^{-1}
    \quad
    (s\downarrow0).
  \end{align*}
  At the same time,
  \begin{align*}
    |p_1-p_2|
    =
    |p_1-p_3|
    =
    2\rho\sin\frac{s}{2}
    \longrightarrow0.
  \end{align*}
  Because
  \begin{align*}
    |p_i-p_j|
    =
    2R\sin\frac{\theta_{ij}}2,
  \end{align*}
  it follows that
  \begin{align*}
    \theta_{12}\to0,
    \quad
    \theta_{13}\to0.
  \end{align*}
  Proposition~\ref{prop:area-below-angle} therefore yields
  \begin{align*}
    \frac{
      \mu_K^{\mathrm{area}}
    }{
      \mu_K^{\mathrm{sr}}
    }
    \longrightarrow1.
  \end{align*}
  Consequently,
  \begin{align}
    \frac{
      \mu_K^{\mathrm{area}}
    }{
      \mu_K^{\mathrm{circ}}
    }
    &=
    \frac{
      \mu_K^{\mathrm{area}}/
      \mu_K^{\mathrm{sr}}
    }{
      \mu_K^{\mathrm{circ}}/
      \mu_K^{\mathrm{sr}}
    }
    \longrightarrow
    \chi.
    \label{eq:area-upper-sharp-family}
  \end{align}
  Thus, the factor $\chi_K$ in the upper bound of \eqref{eq:area-chordal-sharp-comparison} is asymptotically attained for any fixed $\chi>1$.

  Finally, let $\sigma<1$ and $C>0$ be arbitrary. We choose $\chi>1$ sufficiently large that
  \begin{align*}
    \chi^{1-\sigma}>2C.
  \end{align*}
  By \eqref{eq:area-upper-sharp-family}, for this fixed $\chi$ we may choose $s>0$ sufficiently small that
  \begin{align*}
    \frac{
      \mu_K^{\mathrm{area}}
    }{
      \mu_K^{\mathrm{circ}}
    }
    >
    \frac{\chi}{2}.
  \end{align*}
  Because $\chi_K=\chi$,
  \begin{align*}
    \frac{
      \mu_K^{\mathrm{area}}
    }{
      \chi_K^\sigma
      \mu_K^{\mathrm{circ}}
    }
    >
    \frac12
    \chi^{1-\sigma}
    >
    C.
  \end{align*}
  Because $C>0$ is arbitrary,
  \begin{align*}
    \sup_K
    \frac{
      \mu_K^{\mathrm{area}}
    }{
      \chi_K^\sigma
      \mu_K^{\mathrm{circ}}
    }
    =
    \infty.
  \end{align*}
  Therefore, the exponent one in \eqref{eq:area-chordal-sharp-comparison} is optimal.
\end{proof}

In particular, under a uniform spherical-locality bound $\chi_K\leq\chi_0$, the circumradius parameter $\mu_K^{\mathrm{circ}}$, the exact spherical semi-regularity parameter $\mu_K^{\mathrm{sr}}$, and the area-based parameter $\mu_K^{\mathrm{area}}$ are uniformly quantitatively equivalent. Nevertheless, the three parameters encode different geometric information at finite curvature scale.

\begin{remark}[Local flat limit of the semi-regularity parameters]
  \label{rem:local-flat-semiregularity}
  Suppose that a uniformly chordally semi-regular refining family satisfies
  \begin{align*}
    q_h\longrightarrow0.
  \end{align*}
  Corollary~\ref{cor:refining-locality} gives
  \begin{align*}
    \max_K\chi_K
    \longrightarrow1.
  \end{align*}
  Theorem~\ref{thm:sharp-semiregularity-comparison} then yields, uniformly over the mesh,
  \begin{align*}
    \frac{
      \mu_K^{\mathrm{circ}}
    }{
      \mu_K^{\mathrm{sr}}
    }
    \longrightarrow1.
  \end{align*}
  Because $h_{T_K}$ is a longest chordal edge of $T_K$, every chordal edge satisfies
  \begin{align*}
    |p_i-p_j|
    \leq
    h_{T_K}.
  \end{align*}
  Since $q_h\to0$, the identity
  \begin{align*}
    |p_i-p_j|
    =
    2R\sin\frac{\theta_{ij}}2
  \end{align*}
  implies that all central side angles tend to zero uniformly over the mesh. In particular,
  \begin{align*}
    \theta_{12},\theta_{13}
    \longrightarrow0
  \end{align*}
  uniformly. The flat-limit argument in the proof of Proposition~\ref{prop:area-below-angle} is uniform with respect to $c=\cos\alpha_K\in[-1,1]$, and therefore
  \begin{align*}
    \frac{
      \mu_K^{\mathrm{area}}
    }{
      \mu_K^{\mathrm{sr}}
    }
    \longrightarrow1
  \end{align*}
  uniformly over the mesh. Consequently,
  \begin{align*}
    \mu_T
    =
    \mu_K^{\mathrm{circ}}
    \sim
    \mu_K^{\mathrm{sr}}
    \sim
    \mu_K^{\mathrm{area}}
  \end{align*}
  uniformly in the local flat limit, although these parameters describe distinct geometric quantities at finite curvature scale.
\end{remark}

\section{Implications for spherical finite-element meshes}
\label{sec:fem-interpretation}
The preceding geometric results can now be organised according to the two-stage map from a fixed reference triangle to an exact spherical element. 

Let $\widehat T:=\conv\{(0,0)^{\top},(1,0)^{\top},(0,1)^{\top}\}\subset\mathbb R^2$ be the standard reference triangle and let
\begin{align*}
  \Phi_T:\widehat T\longrightarrow T
\end{align*}
be the affine parametrisation of the chordal core, written as
\begin{align}
  \Phi_T(\hat x)
  =
  p_1+B_T\hat x,
  \label{eq:reference-to-chordal-map}
\end{align}
where the columns of $B_T$ are the two chordal edge vectors issuing from $p_1$. The corresponding exact spherical parametrisation is
\begin{align}
  F_K
  :=
  \Psi_T\circ\Phi_T:
  \widehat T
  \longrightarrow
  K,
  \quad
  F_K(\hat x)
  =
  R\frac{\Phi_T(\hat x)}
          {|\Phi_T(\hat x)|}.
  \label{eq:reference-to-spherical-map}
\end{align}
Its differential factorises as
\begin{align}
  DF_K(\hat x)
  =
  D\Psi_T\bigl(\Phi_T(\hat x)\bigr)B_T.
  \label{eq:two-stage-differential}
\end{align}

This factorisation separates two distinct geometric effects. The affine map $\Phi_T$ carries the anisotropy of the chordal triangle, whereas the radial map $\Psi_T$ carries the curvature-induced distortion between $T$ and $K$. The results of the preceding sections provide precisely the corresponding controls. The intrinsic vertex condition $t_K^{\mathrm{sr}}\geq c_{\mathrm{sr}}$ yields uniform chordal semi-regularity through the sharp comparison
\begin{align*}
  \mu_{T_K}
  =
  \mu_K^{\mathrm{circ}}
  \leq
  \frac{2}{\sqrt3}\mu_K^{\mathrm{sr}},
\end{align*}
while relative refinement controls the radial stage through $\vartheta_{T_K}$ and $\chi_{T_K}$. In particular, the singular-value and Jacobian estimates for $D\Psi_T$ obtained in Section~\ref{sec:geometry} become asymptotically flat when $\chi_{T_K}\to1$.

Thus, the intrinsic spherical maximum-angle condition controls the chordal geometry required by the affine stage, whereas spherical locality controls the distortion introduced by the exact radial parametrisation. The following corollary packages these two controls for use in future anisotropic analysis.

\begin{corollary}[Uniform geometric control for spherical finite-element meshes]
  \label{cor:uniform-spherical-mesh-control}
  \label{prop:fem-mesh-criterion}
  \label{prop:intrinsic-spherical-mesh-condition}
  Let $\{\mathscr K_h\}$ be non-degenerate exact spherical triangulations and
  assume that, for some $c_{\mathrm{sr}}>0$,
  \begin{align}
    \inf_h\inf_{K\in\mathscr K_h}t_K^{\mathrm{sr}}
    \geq c_{\mathrm{sr}}.
    \label{eq:partII-intrinsic-geometric-condition}
  \end{align}
  Then,
  \begin{align}
    \sup_h\max_{K\in\mathscr K_h}\mu_{T_K}
    \leq
    \frac{4}{\sqrt3\,c_{\mathrm{sr}}}.
    \label{eq:partII-uniform-chordal-control}
  \end{align}
  If, in addition, the relative chordal diameter $q_h$ defined in
  \eqref{eq:relative-chordal-diameter} satisfies
  \begin{align}
    q_h\longrightarrow0,
    \label{eq:partII-relative-refinement}
  \end{align}
  then
  \begin{align}
    \max_{K\in\mathscr K_h}\vartheta_{T_K}
    \leq
    \frac{q_h}{\sqrt3\,c_{\mathrm{sr}}}.
    \label{eq:partII-uniform-locality-control}
  \end{align}
  Consequently,
  \begin{align}
    \max_{K\in\mathscr K_h}\vartheta_{T_K}
    \longrightarrow0,
    \quad
    \max_{K\in\mathscr K_h}\chi_{T_K}
    \longrightarrow1.
    \label{eq:partII-uniform-radial-limit}
  \end{align}
\end{corollary}

\begin{proof}
  From Theorem~\ref{thm:spherical-sr-maximum-angle-identity},
  \begin{align*}
    \mu_K^{\mathrm{sr}}
    =
    \frac{2}{t_K^{\mathrm{sr}}}
    \leq
    \frac{2}{c_{\mathrm{sr}}}
    \quad
    \forall K\in\mathscr K_h.
  \end{align*}
  Corollary~\ref{cor:spherical-implies-chordal} therefore gives
  \begin{align*}
    \mu_{T_K}
    =
    \mu_K^{\mathrm{circ}}
    \leq
    \frac{2}{\sqrt3}\,
    \mu_K^{\mathrm{sr}}
    \leq
    \frac{4}{\sqrt3\,c_{\mathrm{sr}}},
  \end{align*}
  uniformly in $h$ and $K\in\mathscr K_h$. This proves \eqref{eq:partII-uniform-chordal-control}.

  The exact circumradius identity gives
  \begin{align*}
    \vartheta_{T_K}
    =
    \frac{H_{T_K}}{4R}
    =
    \frac{\mu_{T_K}}{4}
    \frac{h_{T_K}}{R}.
  \end{align*}
  Because
  \begin{align*}
    \frac{h_{T_K}}{R}
    \leq
    q_h,
  \end{align*}
  the bound \eqref{eq:partII-uniform-chordal-control} yields
  \begin{align*}
    \vartheta_{T_K}
    \leq
    \frac{1}{4}
    \frac{4}{\sqrt3\,c_{\mathrm{sr}}}
    q_h
    =
    \frac{q_h}{\sqrt3\,c_{\mathrm{sr}}}.
  \end{align*}
  Taking the maximum over $K\in\mathscr K_h$ proves \eqref{eq:partII-uniform-locality-control}. Therefore,
  \begin{align*}
    \max_{K\in\mathscr K_h}\vartheta_{T_K}
    \longrightarrow0.
  \end{align*}

  Finally,
  \begin{align*}
    \chi_{T_K}
    =
    \left(
      1-\vartheta_{T_K}^2
    \right)^{-1/2}.
  \end{align*}
  Because $s\mapsto(1-s^2)^{-1/2}$ is increasing on $[0,1)$,
  \begin{align*}
    \max_{K\in\mathscr K_h}\chi_{T_K}
    =
    \left(
      1-
      \left(
        \max_{K\in\mathscr K_h}\vartheta_{T_K}
      \right)^2
    \right)^{-1/2}
    \longrightarrow1.
  \end{align*}
  This proves \eqref{eq:partII-uniform-radial-limit}.
\end{proof}

\begin{remark}[Geometric interface for future anisotropic analysis]
  \label{rem:partII-geometric-interface}
  The two conclusions of Corollary~\ref{cor:uniform-spherical-mesh-control} correspond precisely to the two stages of the parametrisation
   \begin{align*}
    \widehat T
    \xrightarrow{\ \Phi_{T_K}\ }
    T_K
    \xrightarrow{\ \Psi_{T_K}\ }
    K.
  \end{align*}
  The intrinsic vertex condition \eqref{eq:partII-intrinsic-geometric-condition} provides uniform chordal semi-regularity for the affine stage through \eqref{eq:partII-uniform-chordal-control}. Relative refinement, $q_h\to0$, controls the radial stage through \eqref{eq:partII-uniform-locality-control} and gives radial-distortion factors converging uniformly to one.

  Thus, the intrinsic spherical maximum-angle condition and relative refinement provide, respectively, the chordal maximum-angle control for the affine stage and the curvature-localisation control required by the two-stage geometric description. This geometric interface is intended for use in future anisotropic analysis. No interpolation estimate, and no claim that these geometric conditions are necessary for convergence of a particular
  finite-element method, is made here.
\end{remark}

\section{Concluding remarks}
\label{sec:conclusion}
We have obtained an exact vertex-based characterisation of the intrinsic maximum-angle condition for exact spherical triangles. The identity
\begin{align*}
  t_K^{\mathrm{sr}}=\sin\alpha_K
\end{align*}
shows that the largest spherical interior angle is determined by the three vertex vectors alone, so that the intrinsic maximum-angle condition may be tested without reference to the radial-locality parameter. The criterion does not require a minimum-angle condition and therefore admits anisotropic elements. The auxiliary chordal and intrinsic circumradius parameters are linked exactly, and the global estimate
\begin{align*}
  \mu_K^{\mathrm{circ}}
  \leq
  \frac{2}{\sqrt3}\mu_K^{\mathrm{sr}}
\end{align*}
is sharp and requires no spherical-locality assumption. The area-based parameter satisfies
\begin{align*}
  \mu_K^{\mathrm{area}}
  <
  \mu_K^{\mathrm{sr}},
\end{align*}
with sharp constant one in the local flat limit.

For spherical finite-element meshes, the intrinsic lower bound
\begin{align*}
  t_K^{\mathrm{sr}}
  \geq
  c_{\mathrm{sr}}>0
\end{align*}
gives the explicit chordal estimate
\begin{align*}
  \mu_{T_K}
  \leq
  \frac{4}{\sqrt3\,c_{\mathrm{sr}}}.
\end{align*}
Under relative refinement $q_h\to0$, the parameters $\vartheta_{T_K}$ converge uniformly to zero and the radial-distortion factors $\chi_{T_K}$ converge uniformly to one. Together with the preceding chordal estimate, these limits provide the geometric controls that we expect to enter an anisotropic interpolation analysis on exact spherical triangles.

Mathematically, two problems remain open. The first is whether the exponent $2$ in the lower bound
\begin{align*}
  \chi_K^{-2}\mu_K^{\mathrm{circ}}
  <
  \mu_K^{\mathrm{area}}
\end{align*}
can be reduced uniformly. The second is to extend the present vertex-based criterion from the sphere to a smooth surface. A natural approach is to combine local curvature information with the geometry of a suitable surface projection, such as the closest-point projection.

\section*{Declarations}

\noindent\textbf{Conflict of interest.}
The author declares that there is no conflict of interest.

\medskip
\noindent\textbf{Funding.}
No funding was received for conducting this study.

\medskip
\noindent\textbf{Data availability.}
No datasets were generated or analysed during the current study.

\end{document}